\documentclass[11pt]{article}

\usepackage[a4paper,margin=1in]{geometry}
\usepackage{amsmath,amssymb,amsthm,mathtools,bm}
\usepackage{microtype}
\usepackage{xcolor}
\usepackage{hyperref}
\usepackage{algorithm}
\usepackage{algpseudocode}
\usepackage{cite}

 \usepackage{mathtools}
\usepackage{enumitem}
\usepackage{array}
\usepackage{multirow}
\usepackage{graphicx}
\usepackage{subcaption}
\usepackage{epstopdf}
\usepackage{float}
\usepackage{cite}
 \usepackage{hyperref}
\usepackage{booktabs}

 \newtheorem{theorem}{Theorem}[section]
\newtheorem{lemma}[theorem]{Lemma}
\newtheorem{corollary}[theorem]{Corollary}
\newtheorem{remark}[theorem]{Remark}

 \newtheorem{assumption}{Assumption}[section]

\allowdisplaybreaks[2]
\newcommand{\R}{\mathbb R}
\newcommand{\ip}[2]{\langle #1,#2\rangle}
\newcommand{\norm}[1]{\left\lVert #1\right\rVert}
\newcommand{\cL}{\mathcal L}
\newcommand{\cE}{\mathcal E}
\newcommand{\cS}{\mathcal S}
\newcommand{\trans}{\mathsf T}
\newcommand{\Gap}{\mathcal G}
\DeclareMathOperator{\dom}{dom}

\DeclareMathOperator{\argmin}{argmin}

\title{Accelerated primal-dual dynamics and algorithms for strongly convex-concave bilinear saddle point problems}

\author{Yun-xi Zhang\thanks{School of Science, Xihua University, Chengdu, Sichuan, China. E-mail: zhangyunxiuser@163.com.},\quad Yi-bin Xiao\thanks{Department of Mathematics,
    University of Electronic Science and Technology of China,
    Chengdu, Sichuan, China. E-mail:xiaoyb9999@hotmail.com},\quad  Xin He\thanks{School of Science, Xihua University, Chengdu, Sichuan, China. E-mail: hexinuser@163.com.\\This work was supported by the National Natural Science Foundation of China (Grant No. 12601606, 12571318, 12171070), Sichuan Science and Technology Program (Grant No. 2025ZNSFSC0813), and the Talent Introduction Project of Xihua University (Grant No. Z241102).}}
\date{\today}

\date{}

\begin{document}

\maketitle

\begin{abstract}
In this paper, we develop accelerated primal-dual dynamics and algorithms for strongly convex-concave saddle point problems with bilinear coupling. Combining strongly convex Nesterov dynamics with scaled heavy-ball dynamics, we propose a second-order primal-dual dynamical system with constant damping and derive an associated algorithm by time discretization and coefficient adjustment. Under suitable conditions on the continuous scaling function $\beta(t)$ and the discrete scaling sequence $\beta_k$, we establish primal-dual gap estimates of $\mathcal O(1/\beta(t))$ in continuous time and $\mathcal O(1/\beta_k)$ in discrete time. The maximal admissible growth rates of these scales recover the exponential and linear convergence rates of strongly convex Nesterov methods in continuous and discrete time, respectively. We also prove convergence of the  primal-dual trajectory and sequence, even when the dual solution is nonunique. Saddle residual estimates of $\mathcal O(1/\sqrt{\beta(t)})$ and $\mathcal O(1/\sqrt{\beta_k})$ hold under globally Lipschitz gradients in continuous time and with a possibly nonsmooth dual potential in discrete time. Numerical experiments illustrate the theoretical results and the performance of the proposed methods.
\end{abstract}

\section{Introduction}\label{sec:intro}

In this paper, we consider the strongly convex-concave bilinear saddle point problem
\begin{equation}\label{ques_main}
 \min_{x\in\R^n}\max_{y\in\R^m}\cL(x,y)=F(x)+\ip{Ax}{y}-G(y),
\end{equation}
where $A\in\R^{m\times n}$, $F$ is $\mu$-strongly convex with $\mu>0$, and $G$ is convex. This partial strong convexity assumption arises naturally in regularized models in image reconstruction, inverse problems, and statistical learning \cite{ChambollePAN,BoydFTML,ChambollePJMIV,LinAcc,HeHFCNSNS}. In such models, a strongly convex regularization term is imposed on the primal variable, whereas the dual potential need not be strongly convex. Problem~\eqref{ques_main} includes several important classes of convex optimization problems. For example, taking $G(y)=\ip{b}{y}$ gives the linearly constrained strongly convex problem \begin{equation}\label{ques_cons} 
		\min_{x\in\R^n} F(x) \quad\text{s.t.}\quad Ax=b. 
	\end{equation}
 More generally, for a proper, closed, convex function $H$, Fenchel duality relates the composite minimization problem \[ \min_{x\in\R^n} F(x)+H(Ax) \] to the saddle formulation~\eqref{ques_main} with $G=H^*$, where $H^*$ denotes the Fenchel conjugate of $H$. These formulations motivate the study of accelerated primal-dual dynamics and algorithms that exploit the strong convexity of $F$ while allowing $G$ to remain merely convex.

\subsection{Accelerated methods in convex optimization}

 The heavy-ball with friction (HBF) dynamical system introduced by Polyak \cite{PolyakCMMP} is a classical model for convex optimization. For a differentiable convex function $f$ with a nonempty minimizer set, it reads
\begin{equation*}
 \ddot x(t)+\gamma\dot x(t)+\nabla f(x(t))=0,\qquad\gamma>0.
\end{equation*}
Its discretizations lead to heavy-ball inertial accelerated algorithms. For convex objectives with Lipschitz gradients, Ghadimi et al.\ \cite{GhadimiFJECC} obtained an $\mathcal O(1/k)$ objective residual bound for averaged iterates and linear convergence under strong convexity, with suitable parameters. Under coercivity, Sun et al.\ \cite{SunAAAI} established a non-ergodic $\mathcal O(1/k)$ bound with a constant stepsize, and linear convergence under a relaxed strong convexity condition. Introducing a positive scale of (HBF) gives
\begin{equation}\label{eq:intro-hb}
 \ddot x(t)+\gamma\dot x(t)+\beta(t)\nabla f(x(t))=0.
\end{equation}
If  $\beta(t)$ is positive and continuously
differentiable, and $\dot\beta(t)\le\gamma\beta(t)$,  the objective residual enjoys the rate of $\mathcal O(1/\beta(t))$, and $\beta(t)=\beta_0e^{rt}$ with $0<r\le\gamma$ gives exponential decay without strong convexity \cite{Balhag,HeHFTAC}. For related time-scaled dynamics and proximal discretizations, see \cite{LuoCMP,AttouchCRSJO,WibisonoWJPNAS,AttouchMOR}.

For a $\mu$-strongly convex function $f$, the damping can instead be chosen in
terms of the strong convexity constant. The continuous-time counterpart of Nesterov's strongly convex method is

\begin{equation}\label{eq:intro-nag}
 \ddot x(t)+2\sqrt\mu\,\dot x(t)+\nabla f(x(t))=0.
\end{equation}
This is an (HBF) dynamic with damping $2\sqrt{\mu}$, for which $f(x(t))-\min f=\mathcal O(e^{-\sqrt{\mu}t})$ \cite{LuoCMP,WilsonRJJMLR,Siegel}. If $\nabla f$ is $L$-Lipschitz, Nesterov's discrete inertial accelerated method gives $\mathcal O((1-\sqrt{\mu s})^k)$ for $0<s\le1/L$ \cite{NesterovBook,LinAcc}. Its gradient is evaluated at an extrapolated point, so its update differs from the classical heavy-ball iteration. Strong convexity therefore improves the quadratic polynomial convergence rates for convex accelerated dynamics and algorithms to exponential and linear rates, respectively \cite{SuBCJMLR,NesterovBook,LuoCMP,WibisonoWJPNAS}.

These models have also been extended to primal-dual methods for the linearly constrained problem \eqref{ques_cons}. Vanishing damping dynamical systems use primal-dual extrapolations to obtain accelerated Lagrangian and feasibility estimates, and motivate discrete augmented Lagrangian methods \cite{HeHFSJCO,BotNJDE,BotCNMP,LuoJOGO,LuoMC,ZhangAMO,DengSiam}. More closely related to our construction, He et al.\cite{HeHFTAC} combined a second-order primal equation and a first-order dual equation with constant damping, and time scaling $\beta(t)$. Suitable scales give an $\mathcal O(1/\beta(t))$ Lagrangian gap bound, objective residual and feasibility violation even for a merely convex objective. Their discrete construction in \cite{HeHFJCAM} gives scaling-dependent objective and feasibility bounds of $\mathcal O(1/\beta_k)$ and convergence of the primal-dual sequence. For strongly convex objectives, He et al. \cite{HeHFAUTO26} incorporated Nesterov damping in the primal equation and a growing dual scale, obtaining $\mathcal O(e^{-\sqrt\mu t})$ objective and feasibility bounds and an accelerated linearized augmented Lagrangian algorithm. These developments motivate combining primal curvature with dual scaling for a general convex dual potential.

\subsection{Primal-dual acceleration for bilinear saddle point problems}
 
Accelerated primal-dual  methods for bilinear saddle point problems have been studied under different  assumptions \cite{TranMPC,LongCnsns,BaiORL,BaiJNM,LinAcc,BaiJSC25} . In the convex-concave setting, Chambolle and Pock \cite{ChambollePJMIV} established an $\mathcal O(1/k)$ ergodic gap rate for a proximal primal-dual algorithm. Their accelerated variant uses strong convexity of one potential to obtain an $\mathcal O(1/k^2)$ estimate for the squared distance of the corresponding variable. Under semi-strong convexity, several methods achieve $\mathcal O(1/k^2)$ gap estimates. Examples include the non-ergodic accelerated method in \cite{HeHFCNSNS}, the splitting methods in \cite{LuoArx24}, the accelerated golden-ratio method with an ergodic gap bound in \cite{ChangJSC}, and the accelerated smoothed-gap framework in \cite{TranDinhOL}. These results show that strong convexity in one variable can be exploited to accelerate primal-dual methods. For further accelerated primal-dual algorithms under various convexity and smoothness assumptions, we refer to \cite{ChenLOSJO,KovalevGR,ThekPMLR,KhalafiBPMLR,CondatSR26} and the references therein.

Dynamical systems approaches provide a complementary way to construct and analyze accelerated primal-dual methods. In the convex-concave setting, Nesterov dynamics with vanishing damping $\alpha/t$ have been extended to saddle point problems by introducing primal-dual extrapolations to control the bilinear coupling in the energy estimates \cite{LiTac,ZengLCIFAC,SunJOTA,WangArx,LuoOpt}. Within this framework, He et al.\cite{HeHFAMO} studied second-order primal-dual dynamics with damping, scaling, and extrapolation parameters, and derived estimates for the gap and velocities. Ding et al.\ \cite{DingFVCOAP} analyzed related dynamics with vanishing damping and their discretizations, including an implicit algorithm with a non-ergodic $\mathcal O(1/k^2)$ gap rate. More recently, He and Fang \cite{HeFarXiv26} established convergence of the entire primal-dual trajectory in the differentiable convex-concave setting with damping $\alpha/t$, $\alpha\ge3$. For $\alpha>3$, they obtained an $o(1/t^2)$ gap estimate and an $o(1/t)$ velocity estimate. When both potentials are strongly convex, constant damping can be used to obtain exponential convergence rates. He and Fang \cite{HeAuto27} coupled two strongly convex Nesterov dynamics, with damping coefficients determined by the strong convexity constants of $F$ and $G$. Compatible primal-dual extrapolations yield the exponential gap rate. Their discrete method achieves the linear convergence rate. In this setting, the gap controls the squared distances of both variables to the unique saddle point.
 
 When only $F$ is strongly convex, however, the gap no longer controls the dual distance, and the constant damping construction in \cite{HeAuto27} does not directly apply. Although the convex-concave analyses in \cite{DingFVCOAP,HeFarXiv26} cover this setting, their polynomial estimates do not exploit the additional primal strong convexity. This raises the question of whether the exponential and linear gap rates can be retained with strong convexity on only one side, together with convergence of both variables. To address this question, we combine strongly convex Nesterov dynamics for the primal variable with scaled heavy-ball dynamics for the dual variable. Suitable primal-dual extrapolations allow these two mechanisms to share a common energy estimate, while a further analysis establishes convergence of the full trajectory and sequence even when the dual solution is nonunique.

\subsection{Main contributions}

Combining strongly convex Nesterov dynamics \eqref{eq:intro-nag} with scaled heavy-ball dynamics \eqref{eq:intro-hb}, we develop accelerated primal-dual methods that exploit primal strong convexity through damping, dual scaling, and velocity extrapolation. Our main contributions are as follows.

\begin{enumerate}[label=(\roman*),leftmargin=*]

\item
\textbf{Accelerated continuous-time primal-dual dynamical system.} We propose a second-order primal-dual dynamical system with constant damping and dual scaling $\beta(t)$. For a positive, continuously differentiable, nondecreasing scale satisfying $\dot\beta(t)\le\sqrt{\mu}\beta(t)$ and $\beta(t)\to+\infty$, we establish global well-posedness, an $\mathcal O(1/\beta(t))$ primal-dual gap estimate, an $\mathcal O(1/\sqrt{\beta(t)})$ primal distance estimate, and convergence of the entire trajectory to a saddle point. The saddle residual satisfies the same $\mathcal O(1/\sqrt{\beta(t)})$ estimate under globally Lipschitz gradients. Exponential scaling yields the gap rate $\mathcal O(e^{-\sqrt{\mu}t})$, improving on the polynomial continuous-time gap estimates in \cite{DingFVCOAP,HeFarXiv26} by exploiting primal strong convexity, without requiring dual strong convexity as in \cite{HeAuto27}.

\item
\textbf{Inertial accelerated primal-dual algorithm.}
We derive an inertial accelerated primal-dual algorithm by time discretization and coefficient adjustment. For smooth strongly convex $F$ and possibly nonsmooth convex $G$, suitable stepsize and scaling conditions yield a non-ergodic $\mathcal O(1/\beta_k)$ primal-dual gap estimate, $\mathcal O(1/\sqrt{\beta_k})$ primal distance and saddle residual estimates, and convergence of the entire sequence, even when the dual solution is nonunique. Geometric scaling gives the linear gap rate $\mathcal O((1-\sqrt{\mu s})^k)$, compared with the polynomial rate under partial strong convexity in \cite{HeHFCNSNS,LuoArx24,TranDinhOL,ChangJSC,DingFVCOAP}. Unlike \cite{HeAuto27}, this conclusion requires neither smoothness nor strong convexity of $G$. 
\end{enumerate}

\subsection{Organization of the paper}
Section~\ref{sec:conti} develops the dynamical system and its convergence results. Section~\ref{sec:discrete} derives the discrete algorithm and proves the corresponding gap estimates, residual bounds, and iterate convergence. Section~\ref{sec:num} presents numerical experiments to illustrate the theoretical results and evaluate the performance of the proposed methods, and Section~\ref{sec:conclusion} concludes the paper.

\section{Continuous-Time Dynamical System}\label{sec:conti}

In this section, we construct an inertial accelerated primal-dual dynamical system for \eqref{ques_main} and establish its energy estimate, global existence, convergence rates, and trajectory convergence.

\subsection{Preliminaries and dynamical systems}

We begin with the assumptions on the objective functions and the properties of saddle points needed for the subsequent analysis.

\begin{assumption}\label{ass:cont}
The function $F:\R^n\to\R$ is $\mu$-strongly convex with
$\mu>0$, and $G:\R^m\to\R$ is convex. Both functions are
continuously differentiable with locally Lipschitz gradients. The saddle set  of~\eqref{ques_main} is nonempty.
\end{assumption}

All vector norms are Euclidean, matrix norms are the
corresponding induced operator norms, and $A^\trans$
denotes the transpose of $A$. A pair $(x^*,y^*)\in\R^n\times\R^m$ is a saddle point
of $\mathcal L$ in \eqref{ques_main} if
\begin{equation}\label{eq:S-ineqs}
 \mathcal L(x^*,y)\le\mathcal L(x^*,y^*)
 \le\mathcal L(x,y^*)
 \qquad
 \forall (x,y)\in\R^n\times\R^m.
\end{equation}
We denote the set of saddle points by $\cS$ and assume
throughout this section that $\cS\ne\varnothing$.
Under Assumption~\ref{ass:cont}, these inequalities are
equivalent to
\begin{equation}\label{eq:kkt1}
 \nabla F(x^*)+A^\trans y^*=0,
 \qquad
 \nabla G(y^*)-Ax^*=0.
\end{equation}
Strong convexity of $F$ ensures that all saddle points
have the same primal component $x^*$. Hence
\[
 \cS=\{x^*\}\times Y^*,
 \qquad
 Y^*:=\{y\in\R^m:(x^*,y)\in\cS\}.
\]
The dual component need not be unique. For any $u,v\in Y^*$,
\eqref{eq:kkt1} gives $A^\trans(u-v)=0$ and
$\nabla G(u)=\nabla G(v)=Ax^*$. By convexity of $G$,
\[
 \ip{Ax^*}{u-v}
 =\ip{\nabla G(v)}{u-v}
 \le G(u)-G(v)
 \le\ip{\nabla G(u)}{u-v}
 =\ip{Ax^*}{u-v}.
\]
Consequently,
\begin{equation}\label{eq:Sdiff}
 A^\trans(u-v)=0,
 \qquad
 G(u)-G(v)=\ip{Ax^*}{u-v}.
\end{equation}
These identities will be used to compare energies associated with different dual saddle points.

For $(x^*,y^*)\in\cS$, define the fixed primal-dual gap
\begin{equation}\label{eq:Sgap}
 \Gap(x,y):=\mathcal L(x,y^*)-\mathcal L(x^*,y).
\end{equation}
Its primal and dual contributions are
\begin{equation}\label{eq_defDFG}
 \begin{split}
 D_F(x)
 &:=F(x)-F(x^*)+\ip{A(x-x^*)}{y^*},\\
 D_G(y)
 &:=G(y)-G(y^*)-\ip{Ax^*}{y-y^*}.
 \end{split}
\end{equation}
By \eqref{eq:kkt1} and Assumption \ref{ass:cont},
$D_F(x)\ge\frac{\mu}{2}\norm{x-x^*}^2$ and $D_G(y)\ge0$.
Consequently,
\begin{equation}\label{eq:gap}
 \Gap(x,y)=D_F(x)+D_G(y)
 \ge\frac{\mu}{2}\norm{x-x^*}^2.
\end{equation}
Moreover, \eqref{eq:Sdiff} shows that
$D_F$, $D_G$, and $\Gap$ are independent of the choice
of $y^*\in Y^*$.

The asymmetric saddle structure of \eqref{ques_main} guides the construction
of the dynamics. We use the strongly convex Nesterov damping $2\sqrt\mu$ in the primal equation and the scaled heavy-ball mechanism in the dual equation. The primal energy involves $\sqrt\mu(x-x^*)+\dot x$. Matching velocity extrapolations with coefficient $1/\sqrt\mu$ make the bilinear terms cancel in a common energy estimate. This gives the following inertial primal-dual dynamical system:
\begin{equation}\label{eq:ode}
 \left\{
 \begin{aligned}
 &\ddot x(t)+2\sqrt{\mu}\,\dot x(t)+\nabla F(x(t))
   +A^\trans\left(
       y(t)+\frac{1}{\sqrt{\mu}}\dot y(t)
     \right)=0,\\
 &\ddot y(t)+\sqrt{\mu}\,\dot y(t)
   +\beta(t)\left[
       \nabla G(y(t))
       -A\left(
           x(t)+\frac{1}{\sqrt{\mu}}\dot x(t)
         \right)
     \right]=0.
 \end{aligned}
 \right.
\end{equation}
Here $t\ge t_0\ge0$, and the initial positions and velocities are arbitrary. The dual damping is $\sqrt\mu$, and $\beta(t)$ is a positive scale whose conditions are stated below. When $A=0$, the equations reduce to   \eqref{eq:intro-hb} and \eqref{eq:intro-nag}.  

\subsection{Energy estimates and global existence}

We now make the energy balance underlying the construction
precise. Let $x^*$ be the unique primal solution of
\eqref{ques_main}. For any $y^*\in Y^*$, define
\begin{equation}\label{eq:cE}
 \cE_{y^*}(t)
 =\beta(t)\Gap(x(t),y(t))+\frac{\beta(t)}{2}
 \norm{\sqrt{\mu}(x(t)-x^*)+\dot x(t)}^2+\frac12
 \norm{\sqrt{\mu}(y(t)-y^*)+\dot y(t)}^2.
\end{equation}
Here $\Gap(x(t),y(t))$ is the fixed-saddle primal-dual gap defined in \eqref{eq:Sgap}. The weights cancel the coupling terms in the energy derivative. 

\begin{theorem}\label{thm:cE}
Let Assumption~\ref{ass:cont} hold. Let
$\beta:[t_0,\infty)\to(0,\infty)$ be continuously
differentiable and nondecreasing, with
\begin{equation}\label{eq:cscale}
 \dot\beta(t)\le\sqrt{\mu}\,\beta(t)
 \qquad\forall t\ge t_0.
\end{equation}
Then, for arbitrary initial conditions, dynamic \eqref{eq:ode}
has a unique global solution
 $(x(t),y(t))$ with
$x(t)\in C^2([t_0,\infty);\R^n)$ and
$y(t)\in C^2([t_0,\infty);\R^m)$. Moreover, $(x(t),y(t))$ is bounded on $[t_0,+\infty)$, and for
every $y^*\in Y^*$,
\[
 \dot\cE_{y^*}(t)\le 0,\qquad \forall  t\ge t_0.
\]
\end{theorem}

\begin{proof}

We first establish local existence and uniqueness. Writing \eqref{eq:ode} as a first-order system in $(x(t),y(t),\dot x(t),\dot y(t))$, its vector field is continuous in time and locally Lipschitz in the state variables, uniformly on compact time intervals. Indeed, $\nabla F$ and $\nabla G$ are locally Lipschitz, $\beta(t)$ is continuous and locally bounded, and the remaining state-dependent terms are linear. By the Picard--Lindel\"of theorem and the maximal extension theorem \cite[Theorems~2.2 and~2.13]{TeschlBook}, arbitrary initial conditions determine a unique maximal solution on $[t_0,T_{\max})$, where $T_{\max}\in(t_0,\infty]$. Equation~\eqref{eq:ode} implies that $x(t)$ and $y(t)$ are twice continuously differentiable. We derive the energy estimate on this interval and then show that $T_{\max}=\infty$.

Fix $(x^*,y^*)\in\cS$. Subtracting the saddle
conditions \eqref{eq:kkt1} from
\eqref{eq:ode} gives
\begin{equation}\label{eq:diffX}
 \begin{split}
 \frac{d}{dt}
 \big(\sqrt{\mu}(x(t)-x^*)+\dot x(t)\big)
 ={}&
 -\sqrt{\mu}\,\dot x(t)
 -\big(\nabla F(x(t))-\nabla F(x^*)\big)\\
 &-\frac1{\sqrt{\mu}}A^\trans
 \big(\sqrt{\mu}(y(t)-y^*)+\dot y(t)\big),
 \end{split}
\end{equation}
and
\begin{equation*}
 \frac{d}{dt}
 \big(\sqrt{\mu}(y(t)-y^*)+\dot y(t)\big)
 =
 -\beta(t)\big(\nabla G(y(t))-\nabla G(y^*)\big)+\frac{\beta(t)}{\sqrt{\mu}}A
 \big(\sqrt{\mu}(x(t)-x^*)+\dot x(t)\big).
\end{equation*}
By the definitions of $D_F$ and $D_G$ in \eqref{eq_defDFG} and
the saddle conditions \eqref{eq:kkt1},
\[
 \begin{aligned}
 \frac{d}{dt}D_F(x(t))
 &=\ip{\nabla F(x(t))-\nabla F(x^*)}{\dot x(t)},\\
 \frac{d}{dt}D_G(y(t))
 &=\ip{\nabla G(y(t))-\nabla G(y^*)}{\dot y(t)}.
 \end{aligned}
\]
This together with \eqref{eq:diffX} gives
\begin{equation}\label{eq:diff_xsqu}
 \begin{split}
 &\frac12\frac{d}{dt}
 \norm{\sqrt{\mu}(x(t)-x^*)+\dot x(t)}^2=
 -\mu\ip{x(t)-x^*}{\dot x(t)}
 -\sqrt{\mu}\norm{\dot x(t)}^2\\
 &\qquad
 -\sqrt{\mu}
 \ip{x(t)-x^*}{\nabla F(x(t))-\nabla F(x^*)}
 -\frac{d}{dt}D_F(x(t))\\
 &\qquad
 -\frac1{\sqrt{\mu}}
 \ip{A\big(\sqrt{\mu}(x(t)-x^*)+\dot x(t)\big)}
     {\sqrt{\mu}(y(t)-y^*)+\dot y(t)}.
 \end{split}
\end{equation}
Similarly, we have
\begin{equation}\label{eq:diff_ysqu}
 \begin{split}
\frac12\frac{d}{dt}
 \norm{\sqrt{\mu}(y(t)-y^*)+\dot y(t)}^2 &=-\sqrt{\mu}\beta(t)
 \ip{y(t)-y^*}{\nabla G(y(t))-\nabla G(y^*)}
 -\beta(t)\frac{d}{dt}D_G(y(t))\\
 &\quad
 +\frac{\beta(t)}{\sqrt{\mu}}
 \ip{A\big(\sqrt{\mu}(x(t)-x^*)+\dot x(t)\big)}
     {\sqrt{\mu}(y(t)-y^*)+\dot y(t)}.
 \end{split}
\end{equation}

Strong convexity of $F$ and convexity of $G$ give
\[
 \begin{aligned}
 F(x^*)
 &\ge F(x(t))
 +\ip{\nabla F(x(t))}{x^*-x(t)}
 +\frac{\mu}{2}\norm{x(t)-x^*}^2,\\
 G(y^*)
 &\ge G(y(t))
 +\ip{\nabla G(y(t))}{y^*-y(t)}.
 \end{aligned}
\]
Combining these inequalities with \eqref{eq:kkt1} and \eqref{eq_defDFG} ,
we obtain
\begin{equation}\label{eq:c-sc}
 \begin{aligned}
 \ip{x(t)-x^*}{\nabla F(x(t))-\nabla F(x^*)}
 &\ge D_F(x(t))
 +\frac{\mu}{2}\norm{x(t)-x^*}^2,\\
 \ip{y(t)-y^*}{\nabla G(y(t))-\nabla G(y^*)}
 &\ge D_G(y(t)).
 \end{aligned}
\end{equation}

Multiplying \eqref{eq:diff_xsqu} by $\beta(t)$
and adding \eqref{eq:diff_ysqu} cancels the
coupling terms. Since
$\Gap(x(t),y(t))=D_F(x(t))+D_G(y(t))$, we obtain
\begin{equation}\label{eq:c-ene1}
 \begin{aligned}
 &\frac{\beta(t)}2\frac{d}{dt}
 \norm{\sqrt{\mu}(x(t)-x^*)+\dot x(t)}^2
 +\frac12\frac{d}{dt}
 \norm{\sqrt{\mu}(y(t)-y^*)+\dot y(t)}^2\\
 &\quad=
 -\mu\beta(t)\ip{x(t)-x^*}{\dot x(t)}
 -\sqrt{\mu}\beta(t)\norm{\dot x(t)}^2
 -\sqrt{\mu}\beta(t)
 \ip{x(t)-x^*}{\nabla F(x(t))-\nabla F(x^*)}\\
 &\qquad
 -\sqrt{\mu}\beta(t)
 \ip{y(t)-y^*}{\nabla G(y(t))-\nabla G(y^*)}
 -\beta(t)\frac{d}{dt}\Gap(x(t),y(t))\\
 &\quad\le
 -\mu\beta(t)\ip{x(t)-x^*}{\dot x(t)}
 -\sqrt{\mu}\beta(t)\norm{\dot x(t)}^2
 -\frac{\mu\sqrt{\mu}\beta(t)}2
  \norm{x(t)-x^*}^2\\
 &\qquad
 -\sqrt{\mu}\beta(t)\Gap(x(t),y(t))
 -\beta(t)\frac{d}{dt}\Gap(x(t),y(t)),
 \end{aligned}
\end{equation}
where the inequality follows from
\eqref{eq:c-sc}.
The identity
\begin{equation}\label{eq:c-ene2}
 \begin{split}
 &-\mu\ip{x(t)-x^*}{\dot x(t)}
 -\sqrt{\mu}\norm{\dot x(t)}^2
 -\frac{\mu\sqrt{\mu}}2\norm{x(t)-x^*}^2\\
 &\qquad=
 -\frac{\sqrt{\mu}}2
 \norm{\sqrt{\mu}(x(t)-x^*)+\dot x(t)}^2
 -\frac{\sqrt{\mu}}2\norm{\dot x(t)}^2
 \end{split}
\end{equation}
then gives the required  estimate.
Indeed, differentiating \eqref{eq:cE} and using
\eqref{eq:c-ene1} and
\eqref{eq:c-ene2}, we find
\[
 \begin{aligned}
 \dot\cE_{y^*}(t)
 ={}&
 \dot\beta(t)\Gap(x(t),y(t))
 +\beta(t)\frac{d}{dt}\Gap(x(t),y(t))+\frac{\dot\beta(t)}2
 \norm{\sqrt{\mu}(x(t)-x^*)+\dot x(t)}^2\\
 &+\frac{\beta(t)}2\frac{d}{dt}
 \norm{\sqrt{\mu}(x(t)-x^*)+\dot x(t)}^2
 +\frac12\frac{d}{dt}
 \norm{\sqrt{\mu}(y(t)-y^*)+\dot y(t)}^2\\
 \le{}&
 \big(\dot\beta(t)-\sqrt{\mu}\beta(t)\big)
 \left[
 \Gap(x(t),y(t))
 +\frac12
 \norm{\sqrt{\mu}(x(t)-x^*)+\dot x(t)}^2
 \right]-\frac{\sqrt{\mu}\beta(t)}2\norm{\dot x(t)}^2\\
  \le{}& 0,
 \end{aligned}
\]
where the last inequality follows from
\eqref{eq:cscale} and \eqref{eq:gap}.
Consequently,
\begin{equation}\label{eq:DesE}
 0\le\cE_{y^*}(t)\le \cE_{y^*}(t_0)
 \qquad\forall t\in[t_0,T_{\max}).
\end{equation}

We next establish bounds independent of time.
Since $\beta$ is nondecreasing,
$\beta(t)\ge\beta(t_0)>0$.
By \eqref{eq:cE}, \eqref{eq:gap}, and
\eqref{eq:DesE},
\[
 \norm{x(t)-x^*}^2
 \le\frac{2\cE_{y^*}(t_0)}{\mu\beta(t_0)},
 \qquad
 \norm{\sqrt{\mu}(x(t)-x^*)+\dot x(t)}^2
 \le\frac{2\cE_{y^*}(t_0)}{\beta(t_0)}.
\]
Hence $x(t)$ and $\dot x(t)$ are uniformly bounded on
$[t_0,T_{\max})$.

For the dual variable, expanding the square and together with \eqref{eq:cE} and \eqref{eq:DesE} gives
\[
\frac{d}{dt}\norm{y(t)-y^*}^2
 +\sqrt{\mu}\norm{y(t)-y^*}^2=\frac1{\sqrt{\mu}}
 \norm{\sqrt{\mu}(y(t)-y^*)+\dot y(t)}^2
 -\frac1{\sqrt{\mu}}\norm{\dot y(t)}^2\le\frac{2\cE_{y^*}(t_0)}{\sqrt{\mu}}.
\]
Multiplying by $e^{\sqrt{\mu}(t-t_0)}$ and
integrating yields
\[
 \begin{aligned}
 \norm{y(t)-y^*}^2
 &\le e^{-\sqrt{\mu}(t-t_0)}
       \norm{y(t_0)-y^*}^2
 +\frac{2\cE_{y^*}(t_0)}{\mu}
  \big(1-e^{-\sqrt{\mu}(t-t_0)}\big)\\
 &\le\max\left\{
 \norm{y(t_0)-y^*}^2,\frac{2\cE_{y^*}(t_0)}{\mu}
 \right\}.
 \end{aligned}
\]
Thus $y(t)$ is uniformly bounded. Moreover,
\[
 \norm{\dot y(t)}
 \le
 \norm{\sqrt{\mu}(y(t)-y^*)+\dot y(t)}
 +\sqrt{\mu}\norm{y(t)-y^*},
\]
so $\dot y(t)$ is uniformly bounded as well.

It remains to prove that $T_{\max}=\infty$. Suppose, to the contrary, that $T_{\max}<\infty$. The function $\beta$ is bounded on $[t_0,T_{\max}]$, and the state $(x(t),y(t),\dot x(t),\dot y(t))$ is bounded by the preceding estimates. Since the gradients are continuous, \eqref{eq:ode} also bounds $\ddot x(t)$ and $\ddot y(t)$ on $[t_0,T_{\max})$. The full state is therefore Lipschitz on this interval and has a finite limit as $t\uparrow T_{\max}$. Local existence with this limit as initial data extends the solution beyond $T_{\max}$, contradicting maximality. Hence $T_{\max}=\infty$. 

The bounds obtained above are independent of time and therefore hold on $[t_0,\infty)$. This proves global existence, uniqueness, boundedness, and energy estimate.
\end{proof}

\subsection{Convergence analysis}

We now establish primal-dual gap bounds and convergence of the trajectory. Globally Lipschitz gradients also yield a saddle residual estimate.

\begin{theorem}\label{thm:convPD}
Let the assumptions of Theorem~\ref{thm:cE} hold, and let $x^*$ be the unique primal solution of \eqref{ques_main}. Suppose additionally that $\beta(t)\to+\infty$ as $t\to\infty$. Then the solution $(x(t),y(t))$ of dynamic \eqref{eq:ode} satisfies the following conclusions.

\begin{enumerate}
\item[(i)] As $t\to\infty$, $x(t)\to x^*$,  and 
for every $y^*\in Y^*$,
\begin{align*}
 \Gap(x(t),y(t))
 &=\mathcal O\left(\frac{1}{\beta(t)}\right), \qquad\norm{x(t)-x^*}^2
 =\mathcal O\left(\frac{1}{\beta(t)}\right),\\
 & \norm{\dot x(t)}+\norm{A^\trans(y(t)-y^*)}
=\mathcal O\left(\frac{1}{\sqrt{\beta(t)}}\right).
\end{align*}

\item[(ii)]
There exists $y^*\in Y^*$ such that
$
 y(t)\to y^*$ as $t\to\infty$.
 
\item[(iii)]
If, in addition, $\nabla F$ and $\nabla G$ are globally
Lipschitz, then
\[
\left\|
 \begin{pmatrix}
  \nabla F(x(t))+A^\trans y(t)\\
  \nabla G(y(t))-Ax(t)
 \end{pmatrix}\right\|
 =\mathcal O\left(\frac{1}{\sqrt{\beta(t)}}\right).
\]
\end{enumerate}
\end{theorem}

\begin{proof}
By Theorem~\ref{thm:cE}, the solution is globally
defined and for any  $y^*\in Y^*$,
\[
 0\le\cE_{y^*}(t)\le\cE_{y^*}(t_0)
 \qquad \forall t\ge t_0.
\]

(i). By \eqref{eq:cE} and
\eqref{eq:gap},
\begin{equation}\label{eq_gapx}
 \Gap(x(t),y(t))
 \le\frac{\cE_{y^*}(t_0)}{\beta(t)},
 \qquad
 \norm{x(t)-x^*}^2
 \le\frac{2\cE_{y^*}(t_0)}{\mu\beta(t)}.
\end{equation}
These inequalities give the first two estimates
in (i). Since $\beta(t)\to+\infty$, they also
imply $x(t)\to x^*$. Moreover,
\begin{equation}\label{eq_dotx}
 \norm{\dot x(t)}
 \le
 \norm{\sqrt{\mu}(x(t)-x^*)+\dot x(t)}
 +\sqrt{\mu}\norm{x(t)-x^*}\le
 2\sqrt{\frac{2\cE_{y^*}(t_0)}{\beta(t)}}
 =\mathcal O\left(\frac{1}{\sqrt{\beta(t)}}\right).
\end{equation}

To estimate $A^\trans(y(t)-y^*)$, multiply the primal
equation in \eqref{eq:ode} by $\sqrt{\mu}$ and use
\eqref{eq:kkt1}. Rearranging gives
\begin{equation}\label{eq:cfilter}
 \begin{split}
 &\frac{d}{dt}
 \big(A^\trans(y(t)-y^*)+\sqrt{\mu}\dot x(t)\big)
 +\sqrt{\mu}
 \big(A^\trans(y(t)-y^*)+\sqrt{\mu}\dot x(t)\big)\\
 &\qquad=
 -\mu\dot x(t)
 -\sqrt{\mu}\big(\nabla F(x(t))-\nabla F(x^*)\big).
 \end{split}
\end{equation}
Choose a closed ball containing $x^*$ and the bounded
primal trajectory. Local Lipschitz continuity of
$\nabla F$ implies Lipschitz continuity on this compact
convex set, with a constant $L$ independent of time.
Consequently,
$\|\nabla F(x(t))-\nabla F(x^*)\|
\le L\|x(t)-x^*\|$ for all $t\ge t_0$. The previously established bounds on
$\norm{x(t)-x^*}$ and $\norm{\dot x(t)}$ therefore give
\begin{equation}\label{eq:cforce}
 \norm{-\mu\dot x(t)
 -\sqrt{\mu}\big(\nabla F(x(t))-\nabla F(x^*)\big)}\le
 \mu\norm{\dot x(t)}+\sqrt{\mu}L\norm{x(t)-x^*}
 \le\frac{C}{\sqrt{\beta(t)}},
\end{equation}
where $C>0$ is independent of $t$.

Multiplying \eqref{eq:cfilter} by $e^{\sqrt{\mu}t}$
and integrating from $t_0$ to $t$, we obtain
\[
 \begin{split}
 &e^{\sqrt{\mu}t}
 \big(A^\trans(y(t)-y^*)+\sqrt{\mu}\dot x(t)\big)-
 e^{\sqrt{\mu}t_0}
 \big(A^\trans(y(t_0)-y^*)+\sqrt{\mu}\dot x(t_0)\big)\\
 &\qquad
= -\int_{t_0}^{t}e^{\sqrt{\mu}s}
 \left[\mu\dot x(s)
 +\sqrt{\mu}\big(\nabla F(x(s))-\nabla F(x^*)\big)\right]\,ds.
  \end{split}
\]
Then, using
\eqref{eq:cforce}, we obtain
\[
\norm{
 e^{\sqrt{\mu}t}
 \big(A^\trans(y(t)-y^*)+\sqrt{\mu}\dot x(t)\big)
 }\le \norm{e^{\sqrt{\mu}t_0}
 \big(A^\trans(y(t_0)-y^*)+\sqrt{\mu}\dot x(t_0)\big)}+C\int_{t_0}^{t}\frac{e^{\sqrt{\mu}s}}{\sqrt{\beta(s)}}\,ds.
\]
By the triangle inequality and division by
$e^{\sqrt{\mu}t}$, this gives
\begin{equation}\label{eq:cvar}
 \norm{A^\trans(y(t)-y^*)+\sqrt{\mu}\dot x(t)}\le
 e^{-\sqrt{\mu}(t-t_0)}
 \norm{A^\trans(y(t_0)-y^*)+\sqrt{\mu}\dot x(t_0)}
+C\int_{t_0}^{t}
 \frac{e^{-\sqrt{\mu}(t-s)}}{\sqrt{\beta(s)}}\,ds.
\end{equation}

Since $\dot\beta(t)/\beta(t)\le\sqrt{\mu}$, it gives
\begin{equation}\label{eq:cratio}
 \frac{\beta(t)}{\beta(s)}
 \le e^{\sqrt{\mu}(t-s)}
 \qquad \forall t\ge s\ge t_0.
\end{equation}
Hence
\[
 \int_{t_0}^{t}
 \frac{e^{-\sqrt{\mu}(t-s)}}{\sqrt{\beta(s)}}\,ds
 \le
 \frac{1}{\sqrt{\beta(t)}}
 \int_{t_0}^{t}e^{-\frac{\sqrt{\mu}}{2}(t-s)}\,ds
 \le\frac{2}{\sqrt{\mu}\sqrt{\beta(t)}}.
\]
Taking $s=t_0$ in \eqref{eq:cratio} also yields
$e^{-\sqrt{\mu}(t-t_0)}
\le\sqrt{\beta(t_0)/\beta(t)}$.
Thus both terms on the right-hand side of
\eqref{eq:cvar} are
$\mathcal O(1/\sqrt{\beta(t)})$. Substituting these two bounds into \eqref{eq:cvar},
we conclude that
\[
\begin{aligned}
 &\norm{A^\trans(y(t)-y^*)+\sqrt{\mu}\dot x(t)}\le\frac{
 \sqrt{\beta(t_0)}
 \norm{A^\trans(y(t_0)-y^*)+\sqrt{\mu}\dot x(t_0)}
 +\frac{2C}{\sqrt{\mu}}
 }{\sqrt{\beta(t)}}=\mathcal O\left(\frac{1}{\sqrt{\beta(t)}}\right).
 \end{aligned}
 \]
Together with \eqref{eq_dotx} implies
\[
 \norm{A^\trans(y(t)-y^*)}
 \le
 \norm{A^\trans(y(t)-y^*)+\sqrt{\mu}\dot x(t)}
 +\sqrt{\mu}\norm{\dot x(t)}=\mathcal O\left(\frac{1}{\sqrt{\beta(t)}}\right),
\]
which completes the proof of (i).

 (ii). From Theorem \ref{thm:cE}, we obtain that $y(t)$ is bounded on $[t_0,+\infty)$. Let $\bar y$ be any cluster
point of $y(t)$, and choose $t_j\to+\infty$ such
that $y(t_j)\to\bar y$. It follows from \eqref{eq:gap} and \eqref{eq_gapx} that
\[
 0\le D_G(y(t))
 \le\Gap(x(t),y(t))
 \le\frac{\cE_{y^*}(t_0)}{\beta(t)}
 \to0.
\]
Thus $D_G(\bar y)=0$ by continuity. Since $D_G$
is nonnegative and differentiable, $\bar y$
is a minimizer and
\[
 \nabla D_G(\bar y)=\nabla G(\bar y)-Ax^*=0.
\]
The adjoint estimate in (i) also implies
$
 A^\trans\bar y=A^\trans y^*=-\nabla F(x^*).
$
Hence $(x^*,\bar y)$ satisfies
\eqref{eq:kkt1}, and every cluster point
of $y(t)$ belongs to $Y^*$.

To show that the cluster point is unique, take
any $u,v\in Y^*$. By
\eqref{eq:Sdiff}, the gap and the
primal part of \eqref{eq:cE} are independent
of the reference dual saddle. Expanding the
difference of the dual squares gives
\begin{equation}\label{eq:cediff}
 \begin{split}
 \cE_u(t)-\cE_v(t)
 ={}&-\sqrt{\mu}
 \ip{\sqrt{\mu}y(t)+\dot y(t)}{u-v}+\frac{\mu}{2}
 \big(\norm{u}^2-\norm{v}^2\big).
 \end{split}
\end{equation}
By Theorem~\ref{thm:cE}, both $\cE_u(t)$ and
$\cE_v(t)$ are nonnegative and nonincreasing,
and hence have finite limits. Equation~\eqref{eq:cediff}
therefore implies that, for each fixed $u,v\in Y^*$,
\begin{equation}\label{eq_limyy}
	\lim_{t\to+\infty}\ip{\sqrt{\mu}y(t)+\dot y(t)}{u-v}=\ell
\end{equation}
for some $\ell\in\R$. To deduce a limit for
$\ip{y(t)}{u-v}$, observe that
\[
 \frac{d}{dt}\ip{y(t)}{u-v}
 +\sqrt{\mu}\ip{y(t)}{u-v}
 =\ip{\sqrt{\mu}y(t)+\dot y(t)}{u-v}.
\]
Subtracting $\ell$ from the scalar equation and
multiplying by $e^{\sqrt{\mu}t}$ gives
\[
 \frac{d}{dt}\left[
 e^{\sqrt{\mu}t}
 \left(\ip{y(t)}{u-v}-\frac{\ell}{\sqrt{\mu}}\right)
 \right]
 =
 e^{\sqrt{\mu}t}
 \left(\ip{\sqrt{\mu}y(t)+\dot y(t)}{u-v}-\ell\right).
\]
Integrating over $[t_0,t]$ and dividing by
$e^{\sqrt{\mu}t}$, we obtain
\begin{equation}\label{eq:ytu}
 \begin{split}
 \ip{y(t)}{u-v}-\frac{\ell}{\sqrt{\mu}}
 ={}&e^{-\sqrt{\mu}(t-t_0)}
 \left(
 \ip{y(t_0)}{u-v}-\frac{\ell}{\sqrt{\mu}}
 \right)\\
 &+\int_{t_0}^{t}e^{-\sqrt{\mu}(t-s)}
 \left(
 \ip{\sqrt{\mu}y(s)+\dot y(s)}{u-v}-\ell
 \right)\,ds.
 \end{split}
\end{equation}
The first term on the right-hand side tends to zero.
To estimate the integral, let $\varepsilon>0$.
By \eqref{eq_limyy}, there exists $T\ge t_0$ such that
\[
 \left|
 \ip{\sqrt{\mu}y(s)+\dot y(s)}{u-v}-\ell
 \right|\le\varepsilon
 \qquad\forall s\ge T.
\]
For any $t\ge T$, split the integral in \eqref{eq:ytu}
over $[t_0,T]$ and $[T,t]$.
On the first interval,
\[
 \begin{split}
 &\left|
 \int_{t_0}^{T}e^{-\sqrt{\mu}(t-s)}
 \left(
 \ip{\sqrt{\mu}y(s)+\dot y(s)}{u-v}-\ell
 \right)\,ds
 \right|\\
 &\quad\le
 e^{-\sqrt{\mu}(t-T)}
 \int_{t_0}^{T}e^{-\sqrt{\mu}(T-s)}
 \left|
 \ip{\sqrt{\mu}y(s)+\dot y(s)}{u-v}-\ell
 \right|\,ds.
 \end{split}
\]
The last integral is finite by continuity and is
independent of $t$. Hence this contribution tends
to zero as $t\to\infty$.
On the second interval,
\[
\left|
 \int_T^t e^{-\sqrt{\mu}(t-s)}
 \left(
 \ip{\sqrt{\mu}y(s)+\dot y(s)}{u-v}-\ell
 \right)\,ds
 \right|\le
 \varepsilon\int_T^t e^{-\sqrt{\mu}(t-s)}\,ds
 =\frac{\varepsilon}{\sqrt{\mu}}
 \left(1-e^{-\sqrt{\mu}(t-T)}\right)
 \le\frac{\varepsilon}{\sqrt{\mu}}.
\]
Taking absolute values in \eqref{eq:ytu} and
letting $t\to\infty$ with $T$ fixed therefore gives
\[
 \limsup_{t\to\infty}
 \left|
 \ip{y(t)}{u-v}-\frac{\ell}{\sqrt{\mu}}
 \right|
 \le\frac{\varepsilon}{\sqrt{\mu}}.
\]
Since $\varepsilon>0$ is arbitrary, we conclude that
\begin{equation}\label{eq:cyproj}
 \lim_{t\to\infty}\ip{y(t)}{u-v}
 =\frac{\ell}{\sqrt{\mu}}.
\end{equation}

We now use this scalar limit to show that the dual trajectory has only one cluster point. Let $u$ and $v$ be any two cluster points of $y(t)$. They belong to $Y^*$ by the preceding argument. Choose sequences $t_j\to\infty$ and $s_j\to\infty$ such that $y(t_j)\to u$ and $y(s_j)\to v$. Applying \eqref{eq:cyproj} along these two sequences gives 
\[ \ip{u}{u-v} =\frac{\ell}{\sqrt{\mu}} =\ip{v}{u-v}. \] 
Subtracting the two inner products yields $\norm{u-v}^2=0$, so $u=v$. Since $y(t)$ is bounded in finite dimensions, it therefore converges to some $y^*\in Y^*$. This completes the proof of (ii).

(iii). Choose positive constants $L_F,L_G>0$ that bound the
global Lipschitz moduli of $\nabla F$ and $\nabla G$,
respectively.  Convexity of $G$ therefore gives $D_G(y)\ge0$
and $D_G(y^*)=0$. Moreover,
\begin{equation}\label{eq_DGdiff}
 \nabla D_G(y)=\nabla G(y)-Ax^*.
\end{equation}
Since the affine term does not affect gradient
differences, for any $u,v\in\R^m$,
\[
 \norm{\nabla D_G(u)-\nabla D_G(v)}
 =\norm{\nabla G(u)-\nabla G(v)}
 \le L_G\norm{u-v}.
\]
Thus $\nabla D_G$ is also globally $L_G$-Lipschitz. Then the descent lemma and \eqref{eq_DGdiff} give
\[
 D_G(v)\le D_G(y(t))
 +\ip{\nabla G(y(t))-Ax^*}{v-y(t)}
 +\frac{L_G}{2}\norm{v-y(t)}^2
 \qquad \forall v\in\R^m.
\]
Choosing
$v=y(t)-\frac{1}{L_G}\big(\nabla G(y(t))-Ax^*\big)$
and using $D_G(v)\ge0$, we obtain
\[
 0\le D_G(v)\le D_G(y(t))
 -\frac{1}{2L_G}\norm{\nabla G(y(t))-Ax^*}^2.
\]
Consequently,
\begin{equation}\label{eq:cgrad}
 \norm{\nabla G(y(t))-Ax^*}^2
 \le2L_GD_G(y(t)).
\end{equation}

To estimate the dual component of the equilibrium
residual, write
\[
 \nabla G(y(t))-Ax(t)
 =\nabla G(y(t))-Ax^*-A(x(t)-x^*).
\]
Using $\norm{a+b}^2\le2\norm{a}^2+2\norm{b}^2$,
\eqref{eq:cgrad}, and the bounds established
in the proof of (i), we obtain
\[
 \begin{aligned}
 \norm{\nabla G(y(t))-Ax(t)}^2
 &\le
 2\norm{\nabla G(y(t))-Ax^*}^2
 +2\norm{A}^2\norm{x(t)-x^*}^2\\
 &\le
 4L_GD_G(y(t))
 +2\norm{A}^2\norm{x(t)-x^*}^2\\
 &\le
 \left(4L_G+\frac{4\norm{A}^2}{\mu}\right)
 \frac{\cE_{y^*}(t_0)}{\beta(t)}.
 \end{aligned}
\]
Here we used $D_G(y(t))\le\Gap(x(t),y(t))$.

For the primal component, \eqref{eq:kkt1}
implies
\[
 \nabla F(x(t))+A^\trans y(t)
 =\nabla F(x(t))-\nabla F(x^*)
 +A^\trans(y(t)-y^*).
\]
Global Lipschitz continuity of $\nabla F$ and
the estimates in (i) therefore yield
\[
 \norm{\nabla F(x(t))+A^\trans y(t)}^2
 \le
 2L_F^2\norm{x(t)-x^*}^2
 +2\norm{A^\trans(y(t)-y^*)}^2=\mathcal O\left(\frac{1}{\beta(t)}\right).
\]
Finally, the squared Euclidean norm of the
equilibrium residual is the sum of the squared
norms of its two components. Hence
\[
 \left\| \begin{pmatrix}
  \nabla F(x(t))+A^\trans y(t)\\
  \nabla G(y(t))-Ax(t)
 \end{pmatrix}\right\|^2=\norm{\nabla F(x(t))+A^\trans y(t)}^2
 +\norm{\nabla G(y(t))-Ax(t)}^2=\mathcal O\left(\frac{1}{\beta(t)}\right).
\]
Taking square roots proves (iii).
 
\end{proof}

The gap estimate in Theorem \ref{thm:convPD} shows how the growth of $\beta(t)$
determines the convergence rate.  Condition \eqref{eq:cscale} bounds the logarithmic growth of $\beta$ by $\sqrt\mu$. Exponential scaling with this exponent gives the fastest gap bound supplied by Theorem~\ref{thm:convPD}; polynomial scales give the following alternatives.

\begin{corollary}\label{cor:cscale}
Let Assumption~\ref{ass:cont} hold. For any $\beta_0>0$, the
following estimates hold along the solution of
\eqref{eq:ode} as $t\to\infty$.
\begin{enumerate}
\item[(i)]
If $\beta(t)=\beta_0e^{\theta t}$ with
$0<\theta\le\sqrt{\mu}$, then
\[
 \Gap(x(t),y(t))=\mathcal O(e^{-\theta t}).
\]
In particular, $\theta=\sqrt{\mu}$ gives
$\Gap(x(t),y(t))=\mathcal O(e^{-\sqrt{\mu}t})$.

\item[(ii)]
If $\beta(t)=\beta_0(t+c)^p$ with
$p>0$ and $c\ge p/\sqrt{\mu}$, then
\[
 \Gap(x(t),y(t))=\mathcal O(1/t^p).
\]
\end{enumerate}
\end{corollary}

\begin{remark}
Choosing $\beta(t)=\beta_0e^{\sqrt{\mu}t}$ in Corollary~\ref{cor:cscale} gives the exponential convergence rate $\Gap(x(t),y(t))=\mathcal O(e^{-\sqrt{\mu}t})$. For comparison, He and Fang \cite{HeAuto27} obtain the rate $\mathcal O(e^{-\min\{\sqrt{\mu_F},\sqrt{\mu_G}\}t})$ for a constant damping primal-dual dynamical system when both $F$ and $G$ are strongly convex, with moduli $\mu_F>0$ and $\mu_G>0$, respectively. Here, strong convexity is required only for $F$: the increasing dual scale allows exponential gap decay while $G$ remains merely convex. Moreover, the dual trajectory converges even when the dual solution is nonunique. Thus the result combines exponential gap decay with convergence of the primal-dual trajectory under partial strong convexity.
\end{remark}

\begin{remark} 
Taking $G(y)=\ip{b}{y}$ reduces \eqref{ques_main} to problem \eqref{ques_cons}, with Lagrangian $\mathcal L(x,y)=F(x)+\ip{Ax-b}{y}$. Consequently, the preceding results provide exponential decay of the Lagrangian gap and convergence of the multiplier trajectory for strongly convex linearly constrained optimization. This connects our analysis with the accelerated primal-dual dynamics in \cite{HeHFAUTO26}, which use a second-order primal equation coupled with a first-order dual equation. Dynamic~\eqref{eq:ode} instead retains inertia in both variables and also accommodates general convex dual potentials. When $A=0$, the two equations decouple. The primal equation becomes the strongly convex Nesterov dynamic \eqref{eq:intro-nag}, and the choice $\beta(t)=\beta_0e^{\sqrt{\mu}t}$ recovers the rate $F(x(t))-F(x^*)=\mathcal O(e^{-\sqrt{\mu}t})$ from \cite{LuoCMP,Siegel,WilsonRJJMLR}. The dual equation becomes the scaled heavy-ball dynamic \eqref{eq:intro-hb}, a special case of the inertial dynamic studied in \cite{Balhag,HeHFTAC}. The estimate $G(y(t))-\min G=\mathcal O(1/\beta(t))$ is consistent with their time-scaling analysis. Hence the proposed coupling preserves the corresponding unconstrained acceleration mechanisms. 
\end{remark}

\section{Accelerated primal-dual algorithms}\label{sec:discrete}

In this section, we develop a discrete primal-dual algorithm motivated by the dynamical system~\eqref{eq:ode}. The discrete model allows a nonsmooth dual potential, while global smoothness of the primal function permits an explicit gradient step.

\begin{assumption}\label{ass:disc}
The function $F:\R^n\to\R$ is differentiable, $\mu$-strongly convex with $\mu>0$, and has a globally $L_F$-Lipschitz gradient. The function $G:\R^m\to(-\infty,+\infty]$ is proper, closed, and convex. The saddle set $\cS\subseteq\R^n\times\dom G$ defined by \eqref{eq:S-ineqs} is nonempty.
\end{assumption}

Under Assumption~\ref{ass:disc}, the saddle inequalities are equivalent to \begin{equation}\label{eq_saddle_dis} \nabla F(x^*)+A^\trans y^*=0, \qquad Ax^*\in\partial G(y^*). \end{equation} Strong convexity of $F$ ensures that all saddle points share the same primal component $x^*$, so $\cS=\{x^*\}\times Y^*$. The subgradient inequalities yield \eqref{eq:Sdiff}. Throughout this section, $D_F$ and $D_G$ are defined in  \eqref{eq_defDFG}. The gap identity and bound in \eqref{eq:gap} remain valid, and the gap is finite for $x\in\R^n$ and $y\in\dom G$.

\subsection{Discretization and the algorithm}

To accommodate a nonsmooth dual potential, we
consider the differential inclusion corresponding
to \eqref{eq:ode}:
\begin{equation*}
 \left\{
 \begin{aligned}
 &\ddot x(t)+2\sqrt{\mu}\dot x(t)+\nabla F(x(t))
 +A^\trans\left(y(t)+\frac{1}{\sqrt{\mu}}\dot y(t)\right)=0,\\
 &0\in\ddot y(t)+\sqrt{\mu}\dot y(t)
 +\beta(t)\left[
 \partial G(y(t))
 -A\left(x(t)+\frac{1}{\sqrt{\mu}}\dot x(t)\right)
 \right].
 \end{aligned}
 \right.
\end{equation*}
With time step $\sqrt s$, set
$
 t_k=t_0+k\sqrt s$ and $\tau=\sqrt{\mu s}.
$ We split the primal damping into two equal terms
and use forward and backward differences,
respectively. For the dual damping we use a
backward difference, while the velocities in
both coupling terms are approximated by forward
differences. Define the predictors
\begin{equation}\label{eq:bars}
 \bar x_k=x_k+\frac{1-\tau}{1+\tau}(x_k-x_{k-1}),
 \qquad
 \bar y_k=y_k+(1-\tau)(y_k-y_{k-1}).
\end{equation}
Evaluating the primal gradient at $\bar x_k$ and the dual subdifferential at $y_{k+1}$ gives the finite-difference scheme

\begin{equation*}
 \left\{
 \begin{aligned}
 &\frac{x_{k+1}-2x_k+x_{k-1}}{s}
 +\sqrt{\mu}\frac{x_{k+1}-x_k}{\sqrt{s}}
 +\sqrt{\mu}\frac{x_k-x_{k-1}}{\sqrt{s}}
+\nabla F(\bar x_k)
 +A^\trans\left(y_k+\frac{y_{k+1}-y_k}{\tau}\right)=0,\\
 &0\in
 \frac{y_{k+1}-2y_k+y_{k-1}}{s}
 +\sqrt{\mu}\frac{y_k-y_{k-1}}{\sqrt{s}}+\beta(t_k)\left[
 \partial G(y_{k+1})
 -A\left(x_k+\frac{x_{k+1}-x_k}{\tau}\right)
 \right].
 \end{aligned}
 \right.
\end{equation*}
Using \eqref{eq:bars}, this becomes
\[
\left\{\begin{aligned}
 x_{k+1}&=\bar x_k-\frac{s}{1+\tau}\left[\nabla F(\bar x_k)+A^\trans\left(y_k+\frac{y_{k+1}-y_k}{\tau}\right)\right],\\
 0&\in y_{k+1}-\bar y_k+s\beta(t_k)\left[\partial G(y_{k+1})-A\left(x_k+\frac{x_{k+1}-x_k}{\tau}\right)\right].
\end{aligned}\right.
\]
We replace the primal force coefficient $s/(1+\tau)$ by $s$ to recover the strongly convex Nesterov update when $A=0$ \cite{NesterovBook}. Writing $\beta_k=s\beta(t_k)$ motivates the dual scale; the algorithm subsequently allows an independently chosen positive sequence $\beta_k$. The adjusted updates are
\begin{subequations}\label{eq:algorithm}
\begin{align}
 x_{k+1}
 &=\bar x_k-s\left[
 \nabla F(\bar x_k)
 +A^\trans\left(y_k+\frac{y_{k+1}-y_k}{\tau}\right)
 \right],
 \label{eq:disc-x}\\
 0&\in y_{k+1}-\bar y_k
 +\beta_k\left[
 \partial G(y_{k+1})
 -A\left(x_k+\frac{x_{k+1}-x_k}{\tau}\right)
 \right].
 \label{eq:disc-y}
\end{align}
\end{subequations}
The primal gradient is explicit, while the dual
potential and the extrapolated coupling are
implicit. The matching coupling terms permit
the cancellation needed in the energy analysis.

The coupled updates \eqref{eq:algorithm} can be implemented by first
solving a dual subproblem and then updating the
primal variable explicitly. Indeed, define
\begin{equation*}
 \widehat x_k
 =x_k+\frac{\bar x_k-x_k}{\tau}
 -\frac{s}{\tau}
 \big(\nabla F(\bar x_k)+A^\trans y_k\big).
\end{equation*}
Equation~\eqref{eq:disc-x} gives
\[
 x_k+\frac{x_{k+1}-x_k}{\tau}
 =\widehat x_k
 -\frac{s}{\tau^2}A^\trans(y_{k+1}-y_k).
\]
Substituting this expression into
\eqref{eq:disc-y} yields
\begin{equation}\label{eq:dopt}
 0\in\partial G(y_{k+1})
 +\frac{y_{k+1}-\bar y_k}{\beta_k}
 -A\widehat x_k
 +\frac{s}{\tau^2}AA^\trans(y_{k+1}-y_k).
\end{equation}
This is the optimality condition of the dual
minimization problem
\[
 y_{k+1}=\argmin_{y\in\R^m}\bigg\{
 G(y)+\frac{1}{2\beta_k}\norm{y-\bar y_k}^2
 -\ip{A\widehat x_k}{y}+\frac{s}{2\tau^2}
 \norm{A^\trans(y-y_k)}^2
 \bigg\}.
\]
For $\beta_k>0$, the affine lower bound of $G$ and the positive quadratic term make the objective coercive and strongly convex. It therefore has a unique minimizer. The resulting method is summarized in Algorithm~\ref{alg:pd}.  
Combining \eqref{eq:dopt} with the primal update recovers \eqref{eq:disc-y}, so the algorithm is equivalent to \eqref{eq:algorithm}.

\begin{algorithm}[htbp]
\caption{Inertial accelerated primal-dual algorithm for problem \eqref{ques_main}.}
\label{alg:pd}

\textbf{Initialization:}
Choose $x_0,x_1\in\R^n$, $y_0,y_1\in\dom G$,
and parameters
$
 0<s\le\frac{1}{L_F}$,
$
 \tau=\sqrt{\mu s}\in(0,1)$,
$
 \beta_k>0$ for all $k\ge0$.

\textbf{For $k=1,2,\ldots$:}
\begin{enumerate}
\item Compute the predictors
\[
 \bar x_k
 =x_k+\frac{1-\tau}{1+\tau}(x_k-x_{k-1}),\quad 
 \bar y_k
 =y_k+(1-\tau)(y_k-y_{k-1})\]
and
\[ \widehat x_k
 =x_k+\frac{\bar x_k-x_k}{\tau}
 -\frac{s}{\tau}
 \big(\nabla F(\bar x_k)+A^\trans y_k\big).
\]

\item Compute $y_{k+1}$ by solving
\[ y_{k+1}=\argmin_{y\in\R^m}\bigg\{
 G(y)+\frac{1}{2\beta_k}\norm{y-\bar y_k}^2
 -\ip{A\widehat x_k}{y}+\frac{s}{2\tau^2}
 \norm{A^\trans(y-y_k)}^2
 \bigg\}.\]

\item Update the primal variable:
\[
 x_{k+1}
 =\bar x_k-s\left[
 \nabla F(\bar x_k)
 +A^\trans\left(y_k+\frac{y_{k+1}-y_k}{\tau}\right)
 \right].
\]
\end{enumerate}

\textbf{Output:}
The primal-dual iterates $(x_k,y_k)$.
\end{algorithm}

The dual subproblem in Algorithm \eqref{alg:pd} is strongly convex and therefore has a unique solution. For many structured choices of $G$ and $A$, it admits a
closed-form solution or reduces to a linear system that can be
solved directly. Under Assumption~\ref{ass:disc} and suitable scaling
conditions, we will establish a non-ergodic primal-dual gap bound
of $\mathcal O(1/\beta_k)$. The fastest gap rate guaranteed by our
analysis is attained with the geometric scale
$\beta_k=\beta_0/(1-\tau)^k$, yielding
$\mathcal O((1-\sqrt{\mu s})^k)$ and matching the convergence rate
of Nesterov's accelerated method for strongly convex optimization \cite{NesterovBook,LinAcc}, while requiring strong convexity only of the primal potential.
 
\subsection{Energy estimates}

We now establish an energy estimate for
Algorithm~\ref{alg:pd}. Let $x^*$ be the unique
primal solution and fix any $y^*\in Y^*$.
For $k\ge1$, define
\begin{equation}\label{eq:dE}
 \cE^{y^*}_k
 =\beta_{k-1}\Gap(x_k,y_k)
 +\frac{\beta_{k-1}}{2s}
 \norm{x_k-(1-\tau)x_{k-1}-\tau x^*}^2+\frac12
 \norm{y_k-(1-\tau)y_{k-1}-\tau y^*}^2,
\end{equation}
where $\Gap$ is the fixed-saddle gap defined
in \eqref{eq:Sgap}.
The shifted differences in \eqref{eq:dE} play
the role of the  velocities in the
continuous energy. As before, the primal terms
and the gap share the same weight, while the
dual term remains unscaled.

\begin{lemma}\label{lem:dE}
Let Assumption~\ref{ass:disc} hold, and let
$\{(x_k,y_k)\}_{k\ge 0}$ be generated by Algorithm~\ref{alg:pd}.
Suppose that
\begin{equation}\label{eq:dscale}
 (1-\tau)\beta_k\le\beta_{k-1}
 \qquad \forall k\ge1.
\end{equation}
Then, for every $y^*\in Y^*$ and $k\ge1$, $\cE^{y^*}_k$ is nonnegative and
\begin{equation*}
 \cE^{y^*}_{k+1}-\cE^{y^*}_k
 \le{}0.
\end{equation*}
\end{lemma}
 
\begin{proof}
Fix $k\ge1$ and $y^*\in Y^*$.
We first derive a one-step estimate for the
primal part and then combine it with the dual update.

By  Assumption~\ref{ass:disc}, $D_F$ is
$\mu$-strongly convex and has an
$L_F$-Lipschitz gradient, with
\begin{equation}\label{eq_DDF}
 \nabla D_F(x)=\nabla F(x)+A^\top  y^*,
 \qquad D_F(x^*)=0.
\end{equation}
The descent lemma at $\bar x_k$ gives
\[
 D_F(x_{k+1})
 \le D_F(\bar x_k)
 +\ip{\nabla D_F(\bar x_k)}{x_{k+1}-\bar x_k}
 +\frac{L_F}{2}\norm{x_{k+1}-\bar x_k}^2.
\]
For any $x\in\R^n$, strong convexity gives
\[
 D_F(x)
 \ge D_F(\bar x_k)
 +\ip{\nabla D_F(\bar x_k)}{x-\bar x_k}
 +\frac{\mu}{2}\norm{x-\bar x_k}^2.
\]
Subtracting the second inequality from the first and using $L_F\le1/s$, we obtain
\begin{equation}\label{eq:done}
 \begin{split}
D_F(x_{k+1})-D_F(x) &\le
 \ip{\nabla D_F(\bar x_k)}{x_{k+1}-x}
 +\frac{L_F}{2}\norm{x_{k+1}-\bar x_k}^2
 -\frac{\mu}{2}\norm{\bar x_k-x}^2\\
 &\le
 \ip{\nabla D_F(\bar x_k)}{x_{k+1}-x}
 +\frac{1}{2s}\norm{x_{k+1}-\bar x_k}^2
 -\frac{\mu}{2}\norm{\bar x_k-x}^2.
 \end{split}
\end{equation}
To express the gradient through the algorithm,
\eqref{eq:disc-x} gives
\[
 \nabla F(\bar x_k)
 =-\frac{x_{k+1}-\bar x_k}{s}
 -A^\trans\left(
 y_k+\frac{y_{k+1}-y_k}{\tau}
 \right).
\]
This together with \eqref{eq_DDF} implies
\begin{equation}\label{eq:dxstep}
 \begin{aligned}
 \nabla D_F(\bar x_k)&=-\frac{x_{k+1}-\bar x_k}{s}
 -A^\trans\left(
 y_k-y^*+\frac{y_{k+1}-y_k}{\tau}
 \right)\\
 &=-\frac{x_{k+1}-\bar x_k}{s}
 -\frac{1}{\tau}A^\trans
 \big(y_{k+1}-(1-\tau)y_k-\tau y^*\big).
 \end{aligned}
\end{equation}
Substituting \eqref{eq:dxstep} into
\eqref{eq:done} yields
\[
 \begin{split}
 D_F(x_{k+1})-D_F(x)
 \le{}&
 -\frac{1}{s}
 \ip{x_{k+1}-\bar x_k}{x_{k+1}-x}
 +\frac{1}{2s}\norm{x_{k+1}-\bar x_k}^2-\frac{\mu}{2}\norm{\bar x_k-x}^2\\
 &-\frac{1}{\tau}
 \ip{A(x_{k+1}-x)}
     {y_{k+1}-(1-\tau)y_k-\tau y^*}.
 \end{split}
\]
Using
$x_{k+1}-x=(x_{k+1}-\bar x_k)+(\bar x_k-x)$
in the first inner product, we obtain, for
every $x\in\R^n$,
\begin{equation}\label{eq:dstep}
 \begin{split}
 D_F(x_{k+1})-D_F(x)
 \le{}&
 -\frac{1}{s}
 \ip{x_{k+1}-\bar x_k}{\bar x_k-x}
 -\frac{1}{2s}\norm{x_{k+1}-\bar x_k}^2-\frac{\mu}{2}\norm{\bar x_k-x}^2\\
 &-\frac{1}{\tau}
 \ip{A(x_{k+1}-x)}
     {y_{k+1}-(1-\tau)y_k-\tau y^*}.
 \end{split}
\end{equation}
Apply \eqref{eq:dstep} with $x=x_k$ and
$x=x^*$, multiply the resulting inequalities
by $1-\tau$ and $\tau$, respectively, and
add them. Since $D_F(x^*)=0$, this gives
\begin{equation}\label{eq:dF}
 \begin{split}
D_F(x_{k+1})-(1-\tau)D_F(x_k) &\le
 -\frac{1}{s}
 \ip{x_{k+1}-\bar x_k}
     {\bar x_k-(1-\tau)x_k-\tau x^*}
 -\frac{1}{2s}\norm{x_{k+1}-\bar x_k}^2\\
 &\quad
 -\frac{\mu}{2}
 \left[
 (1-\tau)\norm{\bar x_k-x_k}^2
 +\tau\norm{\bar x_k-x^*}^2
 \right]\\
 &\quad
 -\frac{1}{\tau}
 \ip{A\big(x_{k+1}-(1-\tau)x_k-\tau x^*\big)}
     {y_{k+1}-(1-\tau)y_k-\tau y^*}.
 \end{split}
\end{equation}

Using the identity
$-\ip{a}{b}
=\tfrac12(\norm{a}^2+\norm{b}^2-\norm{a+b}^2)$,
we can rewrite the first two terms on the
right-hand side of \eqref{eq:dF} as 
\begin{equation}\label{eq:dxkin}
 \begin{split}
 &-\frac{1}{s}
 \ip{x_{k+1}-\bar x_k}
     {\bar x_k-(1-\tau)x_k-\tau x^*}
 -\frac{1}{2s}\norm{x_{k+1}-\bar x_k}^2\\
 &\quad=
 \frac{1}{2s}
 \norm{\bar x_k-(1-\tau)x_k-\tau x^*}^2
 -\frac{1}{2s}
 \norm{x_{k+1}-(1-\tau)x_k-\tau x^*}^2.
 \end{split}
\end{equation}
Recall from \eqref{eq:bars} that
$ \bar x_k-x_k
 =\frac{1-\tau}{1+\tau}(x_k-x_{k-1}),
$
hence
\[
 (1-\tau)(x_k-x_{k-1})
 =(1+\tau)(\bar x_k-x_k),
\]
Therefore
\[
 \begin{aligned}
x_k-(1-\tau)x_{k-1}-\tau x^* &=
 \tau(x_k-x^*)+(1-\tau)(x_k-x_{k-1})\\
 &=
 \tau(x_k-x^*)+(1+\tau)(\bar x_k-x_k)\\
 &=
 \tau(\bar x_k-x^*)+(\bar x_k-x_k).
 \end{aligned}
\]
and then
\begin{equation}\label{eq_barxdiff}
	 \begin{aligned}
 &\bar x_k-(1-\tau)x_k-\tau x^*=
 (1-\tau)(\bar x_k-x_k)
 +\tau(\bar x_k-x^*)\\
 &\quad=
 (1-\tau)
 \big[x_k-(1-\tau)x_{k-1}-\tau x^*
      -\tau(\bar x_k-x^*)\big]
 +\tau(\bar x_k-x^*)\\
 &\quad=
 (1-\tau)
 \big(x_k-(1-\tau)x_{k-1}-\tau x^*\big)
 +\tau^2(\bar x_k-x^*).
 \end{aligned}
\end{equation}
By convexity of the squared norm $\|\cdot\|^2$ and
$\tau^2=\mu s$, from \eqref{eq_barxdiff} we obtain
\begin{equation*}
 \begin{split}
\frac{1}{2s}
 \norm{\bar x_k-(1-\tau)x_k-\tau x^*}^2 &\le\frac{1-\tau}{2s}
 \norm{x_k-(1-\tau)x_{k-1}-\tau x^*}^2
 +\frac{\tau^3}{2s}\norm{\bar x_k-x^*}^2\\
 &=
 \frac{1-\tau}{2s}
 \norm{x_k-(1-\tau)x_{k-1}-\tau x^*}^2
 +\frac{\mu\tau}{2}\norm{\bar x_k-x^*}^2.
 \end{split}
\end{equation*}
Substituting this into
\eqref{eq:dxkin}, then using \eqref{eq:dF}
and discarding the nonpositive term
$-\mu(1-\tau)\norm{\bar x_k-x_k}^2/2$,
we obtain
\begin{equation}\label{eq:dprimal}
 \begin{split}
 &D_F(x_{k+1})
 +\frac{1}{2s}
 \norm{x_{k+1}-(1-\tau)x_k-\tau x^*}^2\\
 &\quad\le
 (1-\tau)\left[
 D_F(x_k)
 +\frac{1}{2s}
 \norm{x_k-(1-\tau)x_{k-1}-\tau x^*}^2
 \right]\\
 &\qquad
 -\frac{1}{\tau}
 \ip{A\big(x_{k+1}-(1-\tau)x_k-\tau x^*\big)}
     {y_{k+1}-(1-\tau)y_k-\tau y^*}.
 \end{split}
\end{equation}

We next estimate the dual part.
By \eqref{eq:disc-y}, there exists
$\eta_{k+1}\in\partial G(y_{k+1})$ such that
\begin{equation}\label{eq:dystep}
 \eta_{k+1}-Ax^*
 =-\frac{y_{k+1}-\bar y_k}{\beta_k}
 +\frac{1}{\tau}
 A\big(x_{k+1}-(1-\tau)x_k-\tau x^*\big).
\end{equation}
Since $\eta_{k+1}-Ax^*\in\partial D_G(y_{k+1})$,
convexity of $D_G$ and $D_G(y^*)=0$ give
\[
 \begin{aligned}
 D_G(y_{k+1})
 &\le\ip{\eta_{k+1}-Ax^*}{y_{k+1}-y^*},\\
 D_G(y_{k+1})-D_G(y_k)
 &\le\ip{\eta_{k+1}-Ax^*}{y_{k+1}-y_k}.
 \end{aligned}
\]
Multiplying these inequalities by $\tau$
and $1-\tau$, respectively, and adding them,
we obtain
\[
 D_G(y_{k+1})-(1-\tau)D_G(y_k)
 \le
 \ip{\eta_{k+1}-Ax^*}
     {y_{k+1}-(1-\tau)y_k-\tau y^*}.
\]
Multiplying by $\beta_k$ and substituting
\eqref{eq:dystep} yields
\begin{equation}\label{eq:dG}
 \begin{split}
 &\beta_k
 \big(D_G(y_{k+1})-(1-\tau)D_G(y_k)\big)\le
 -\ip{y_{k+1}-\bar y_k}
     {y_{k+1}-(1-\tau)y_k-\tau y^*}\\
 &\qquad
 +\frac{\beta_k}{\tau}
 \ip{A\big(x_{k+1}-(1-\tau)x_k-\tau x^*\big)}
     {y_{k+1}-(1-\tau)y_k-\tau y^*}.
 \end{split}
\end{equation}

By the definition of $\bar y_k$ in
\eqref{eq:bars},
\[
 y_{k+1}-\bar y_k
 =\big(y_{k+1}-(1-\tau)y_k-\tau y^*\big)
 -\big(y_k-(1-\tau)y_{k-1}-\tau y^*\big).
\]
Consequently, using the identity
$
 -\ip{a-b}{a}
 =\frac12\norm{b}^2-\frac12\norm{a}^2
 -\frac12\norm{a-b}^2,
$, we obtain
\begin{equation}\label{eq:dykin}
 \begin{split}
 &-\ip{y_{k+1}-\bar y_k}
      {y_{k+1}-(1-\tau)y_k-\tau y^*}=
 \frac12\norm{y_k-(1-\tau)y_{k-1}-\tau y^*}^2
\\ &\qquad\qquad -\frac12\norm{y_{k+1}-(1-\tau)y_k-\tau y^*}^2
-\frac12\norm{y_{k+1}-\bar y_k}^2\\
 &\qquad\quad\le
 \frac12\norm{y_k-(1-\tau)y_{k-1}-\tau y^*}^2
 -\frac12\norm{y_{k+1}-(1-\tau)y_k-\tau y^*}^2.
 \end{split}
\end{equation}
Combining \eqref{eq:dG} and \eqref{eq:dykin}
gives
\begin{equation}\label{eq:ddual}
 \begin{split}
 &\beta_kD_G(y_{k+1})
 +\frac12\norm{y_{k+1}-(1-\tau)y_k-\tau y^*}^2\\
 &\quad\le
 (1-\tau)\beta_kD_G(y_k)
 +\frac12\norm{y_k-(1-\tau)y_{k-1}-\tau y^*}^2\\
 &\qquad
 +\frac{\beta_k}{\tau}
 \ip{A\big(x_{k+1}-(1-\tau)x_k-\tau x^*\big)}
     {y_{k+1}-(1-\tau)y_k-\tau y^*}.
 \end{split}
\end{equation}

Finally, multiplying \eqref{eq:dprimal}
by $\beta_k$ and adding \eqref{eq:ddual}
cancels the coupling terms.
Using $\Gap=D_F+D_G$ and \eqref{eq:dE},
we obtain
\[
\cE^{y^*}_{k+1}-\cE^{y^*}_k\le\big((1-\tau)\beta_k-\beta_{k-1}\big)
 \left[
 \Gap(x_k,y_k)
 +\frac{1}{2s}
 \norm{x_k-(1-\tau)x_{k-1}-\tau x^*}^2
 \right]\le0,
 \]
where the last inequality follows from \eqref{eq:dscale} and  \eqref{eq:gap}.
This completes the proof.
\end{proof}

\subsection{Convergence analysis}

We use the energy estimate to obtain convergence rates for the primal-dual gap, the primal error, and the saddle residual, together with convergence of the dual sequence.

\begin{theorem}\label{thm:dconv}
Suppose that all assumptions of
Lemma~\ref{lem:dE} hold and that
$\beta_k\to+\infty$.
Let $x^*$ be the unique primal solution of
\eqref{ques_main}, and let
$\{(x_k,y_k)\}_{k\ge0}$ be generated by
Algorithm~\ref{alg:pd}.
Then the following conclusions hold.
\begin{enumerate}
\item[(i)]  As $k\to\infty$, $x_k\to x^*$, and
for every $y^*\in Y^*$,
\[
 \Gap(x_k,y_k)
 =\mathcal O\left(\frac{1}{\beta_k}\right),
 \qquad
 \norm{x_k-x^*}^2
 =\mathcal O\left(\frac{1}{\beta_k}\right),
\]
and
\[
 \norm{x_k-x_{k-1}}
 +\norm{A^\trans(y_k-y^*)}
 =\mathcal O\left(\frac{1}{\sqrt{\beta_k}}\right).
\]

\item[(ii)]
There exists $y^*\in Y^*$ such that
$
 \lim_{k\to+\infty}y_k=y^*.
$

\item[(iii)]
As $k\to\infty$,
\[
 \operatorname{dist}\left(
 0,\,
 \begin{pmatrix}
 \nabla F(x_k)+A^\trans y_k\\
 \partial G(y_k)-Ax_k
 \end{pmatrix}
 \right)
 =\mathcal O\left(\frac{1}{\sqrt{\beta_k}}\right).
\]
\end{enumerate}
\end{theorem}
\begin{proof}
\textup{(i)} Fix $y^*\in Y^*$.
By Lemma~\ref{lem:dE},
$0\le\cE^{y^*}_k\le\cE^{y^*}_1$ for all $k\ge1$.
Nonnegativity of the terms in \eqref{eq:dE},
together with \eqref{eq:gap} and
$
 \frac1{\beta_{k-1}}
 \le\frac1{(1-\tau)\beta_k},
$
therefore gives
\[
 \Gap(x_k,y_k)
 =\mathcal O\left(\frac1{\beta_k}\right),
 \qquad
 \norm{x_k-x^*}^2
 =\mathcal O\left(\frac1{\beta_k}\right),
\]
as well as
\[
 \norm{x_k-(1-\tau)x_{k-1}-\tau x^*}
 =\mathcal O\left(\frac1{\sqrt{\beta_k}}\right).
\]
Using the identity
\[
 (1-\tau)(x_k-x_{k-1})
 =\left[x_k-(1-\tau)x_{k-1}-\tau x^*\right]
 -\tau(x_k-x^*)
\]
and $0<\tau<1$, we obtain
\begin{equation}\label{eq:dxrate}
 \norm{x_k-x_{k-1}}
 =\mathcal O\left(\frac1{\sqrt{\beta_k}}\right).
\end{equation}
This together with \eqref{eq:bars} implies
\begin{equation}\label{eq:dxbar}
 \norm{\bar x_k-x^*}\le\norm{x_k-x^*}
 +\frac{1-\tau}{1+\tau}\norm{x_k-x_{k-1}}=\mathcal O\left(\frac1{\sqrt{\beta_k}}\right).
\end{equation}

To estimate $A^\trans(y_k-y^*)$, multiply
\eqref{eq:disc-x} by $\tau/s$ and use
\eqref{eq_saddle_dis}. This gives
\[
 A^\trans(y_{k+1}-y^*)
 =(1-\tau)A^\trans(y_k-y^*)
 -\frac{\tau}{s}(x_{k+1}-\bar x_k)-\tau\big(\nabla F(\bar x_k)-\nabla F(x^*)\big).
\]
Substituting
$
 x_{k+1}-\bar x_k
 =(x_{k+1}-x_k)
 -\frac{1-\tau}{1+\tau}(x_k-x_{k-1})
$
and collecting terms yields
\begin{equation}\label{eq:dfilter}
 \begin{split}
 &A^\trans(y_{k+1}-y^*)
 +\frac{\tau}{s}(x_{k+1}-x_k)=
 (1-\tau)\left[
 A^\trans(y_k-y^*)
 +\frac{\tau}{s}(x_k-x_{k-1})
 \right]\\
 &\qquad
 -\frac{\tau^2(1-\tau)}{s(1+\tau)}
 (x_k-x_{k-1})
 -\tau\big(\nabla F(\bar x_k)-\nabla F(x^*)\big).
 \end{split}
\end{equation}
By \eqref{eq:dxrate}, \eqref{eq:dxbar},
and Lipschitz continuity of $\nabla F$,
\[
 \begin{split}
\left\|
 \frac{\tau^2(1-\tau)}{s(1+\tau)}
 (x_k-x_{k-1})
 +\tau\big(\nabla F(\bar x_k)-\nabla F(x^*)\big)
 \right\| &\le
 \frac{\tau^2(1-\tau)}{s(1+\tau)}
 \norm{x_k-x_{k-1}}
 +\tau L_F\norm{\bar x_k-x^*}\\
 &\le\frac{C}{\sqrt{\beta_k}}
 \end{split}
\]
for some constant $C>0$ independent of $k$.
Iterating \eqref{eq:dfilter} consequently gives
\begin{equation}\label{eq:dfbound}
 \begin{split}
 &\norm{
 A^\trans(y_{k+1}-y^*)
 +\frac{\tau}{s}(x_{k+1}-x_k)}\\
 &\quad\le
 (1-\tau)^k
 \norm{
 A^\trans(y_1-y^*)
 +\frac{\tau}{s}(x_1-x_0)}
 +C\sum_{j=1}^k
 \frac{(1-\tau)^{k-j}}{\sqrt{\beta_j}}.
 \end{split}
\end{equation}
Repeated application of \eqref{eq:dscale} gives
\begin{equation}\label{eq:dratio}
 \frac{\beta_{k+1}}{\beta_j}
 \le\frac1{(1-\tau)^{k+1-j}}
 \qquad\forall\,1\le j\le k+1.
\end{equation}
We use this inequality to estimate both terms
on the right-hand side of \eqref{eq:dfbound}.
For the sum, we have
\[
 \begin{aligned}
 \sum_{j=1}^k
 \frac{(1-\tau)^{k-j}}{\sqrt{\beta_j}}
 &=
 \frac1{\sqrt{\beta_{k+1}}}
 \sum_{j=1}^k(1-\tau)^{k-j}
 \sqrt{\frac{\beta_{k+1}}{\beta_j}}\\
 &\le
 \frac1{\sqrt{1-\tau}\sqrt{\beta_{k+1}}}
 \sum_{j=1}^k(\sqrt{1-\tau})^{k-j}\\
 &\le
 \frac1{
 \sqrt{1-\tau}(1-\sqrt{1-\tau})
 \sqrt{\beta_{k+1}}}.
 \end{aligned}
\]
For the first term of \eqref{eq:dfbound},
\eqref{eq:dratio} with $j=1$ and $\tau\in(0,1)$ implies
\[
 (1-\tau)^k
 \le  (1-\tau)^{k/2}
 \le
 \frac{\sqrt{\beta_1}}{\sqrt{\beta_{k+1}}}.
\]
Substitution into \eqref{eq:dfbound} therefore
yields
\[
 \norm{
 A^\trans(y_{k+1}-y^*)
 +\frac{\tau}{s}(x_{k+1}-x_k)}
 =\mathcal O\left(\frac1{\sqrt{\beta_{k+1}}}\right).
\]
Using the triangle inequality and
\eqref{eq:dxrate} at index $k+1$, we obtain
\[
 \norm{A^\trans(y_{k+1}-y^*)}\le\norm{
 A^\trans(y_{k+1}-y^*)
 +\frac{\tau}{s}(x_{k+1}-x_k)}
 +\frac{\tau}{s}\norm{x_{k+1}-x_k}=\mathcal O\left(\frac1{\sqrt{\beta_{k+1}}}\right).
 \]
Equivalently,
\begin{equation}\label{eq:dadj}
 \norm{A^\trans(y_k-y^*)}
 =\mathcal O\left(\frac1{\sqrt{\beta_k}}\right).
\end{equation}
Together with \eqref{eq:dxrate}, this proves (i).

\medskip
\textup{(ii)} By Lemma~\ref{lem:dE} and
\eqref{eq:dE},
\[
 \norm{y_k-(1-\tau)y_{k-1}-\tau y^*}
 \le\sqrt{2\cE^{y^*}_1}.
\]
Using
\[
 y_k-y^*
 =(1-\tau)(y_{k-1}-y^*)
 +\big(y_k-(1-\tau)y_{k-1}-\tau y^*\big)
\]
and iterating, we obtain
\begin{equation}\label{eq:dybd}
 \norm{y_k-y^*}
 \le
 (1-\tau)^k\norm{y_0-y^*}
 +\sqrt{2\cE^{y^*}_1}
 \sum_{j=0}^{k-1}(1-\tau)^j\le
 \norm{y_0-y^*}
 +\frac{\sqrt{2\cE^{y^*}_1}}{\tau}.
\end{equation}
Thus $\{y_k\}_{k\ge0}$ is bounded.

Let $y_{k_j}\to\bar y$ along a subsequence
with $k_j\to\infty$.
Since $\beta_k\to+\infty$, part~(i) gives
\[
 0\le D_G(y_k)\le\Gap(x_k,y_k)\to0.
\]
Lower semicontinuity of $D_G$ therefore implies
\[
 0\le D_G(\bar y)
 \le\liminf_{j\to\infty}D_G(y_{k_j})=0.
\]
Thus $\bar y$ minimizes $D_G$, so
$
 0\in\partial D_G(\bar y)
 =\partial G(\bar y)-Ax^*.
$
In addition, \eqref{eq:dadj} implies
$A^\trans\bar y=A^\trans y^*=-\nabla F(x^*)$.
By \eqref{eq_saddle_dis}, every cluster point of
$\{y_k\}_{k\ge1}$ belongs to $Y^*$.

To prove uniqueness of the cluster point,
take any $u,v\in Y^*$.
By \eqref{eq:Sdiff}, the gap
and the primal term in \eqref{eq:dE}
are independent of the reference dual solution.
Expanding the difference of the dual squares
gives
\begin{equation}\label{eq:dediff}
 \cE^u_k-\cE^v_k
 =-\tau\ip{y_k-(1-\tau)y_{k-1}}{u-v}
 +\frac{\tau^2}{2}
 \big(\norm{u}^2-\norm{v}^2\big).
\end{equation}
Both energies are nonnegative and nonincreasing,
and hence have finite limits.
It follows from \eqref{eq:dediff} that
\begin{equation}\label{eq:limyk}
 \lim_{k\to\infty}
 \ip{y_k-(1-\tau)y_{k-1}}{u-v}=\ell
\end{equation}
for some $\ell\in\R$.
The identity
\[
 \ip{y_k}{u-v}-\frac{\ell}{\tau}
 =(1-\tau)
 \left(\ip{y_{k-1}}{u-v}-\frac{\ell}{\tau}\right)+\ip{y_k-(1-\tau)y_{k-1}}{u-v}-\ell
\]
allows us to determine the limit of
$\ip{y_k}{u-v}$.
For any fixed $N\ge1$ and $k>N$, iterating
this identity from $N+1$ to $k$ gives
\begin{equation}\label{eq:dyproj}
 \begin{split}
 \ip{y_k}{u-v}-\frac{\ell}{\tau}
 ={}&(1-\tau)^{k-N}
 \left(\ip{y_N}{u-v}-\frac{\ell}{\tau}\right)\\
 &+\sum_{j=N+1}^k(1-\tau)^{k-j}
 \left(
 \ip{y_j-(1-\tau)y_{j-1}}{u-v}-\ell
 \right).
 \end{split}
\end{equation}
Given $\varepsilon>0$, \eqref{eq:limyk}
provides $N\ge1$ such that
\[
 \left|
 \ip{y_j-(1-\tau)y_{j-1}}{u-v}-\ell
 \right|
 \le\tau\varepsilon
 \qquad\forall j>N.
\]
Taking absolute values in \eqref{eq:dyproj}
and summing the geometric series, we obtain
\[
 \begin{aligned}
 \left|\ip{y_k}{u-v}-\frac{\ell}{\tau}\right|
 &\le
 (1-\tau)^{k-N}
 \left|\ip{y_N}{u-v}-\frac{\ell}{\tau}\right|
 +\tau\varepsilon
 \sum_{j=N+1}^k(1-\tau)^{k-j}\\
 &=
 (1-\tau)^{k-N}
 \left|\ip{y_N}{u-v}-\frac{\ell}{\tau}\right|
 +\varepsilon\big(1-(1-\tau)^{k-N}\big).
 \end{aligned}
\]
Since $0<1-\tau<1$, letting $k\to\infty$
with $N$ fixed yields
\[
 \limsup_{k\to\infty}
 \left|\ip{y_k}{u-v}-\frac{\ell}{\tau}\right|
 \le\varepsilon.
\]
As $\varepsilon>0$ is arbitrary, we conclude that
\[
 \lim_{k\to\infty}\ip{y_k}{u-v}=\frac{\ell}{\tau}.
\]

Now let $u$ and $v$ be two cluster points of
$\{y_k\}_{k\ge1}$, both of which belong to $Y^*$.
Evaluating this scalar limit along subsequences
converging to $u$ and $v$ yields
\[
 \ip{u}{u-v}=\ip{v}{u-v}.
\]
Hence $\norm{u-v}^2=0$.
The bounded sequence $\{y_k\}$ therefore has
a unique cluster point and converges to some
$y^*\in Y^*$.
Since part~(i) also gives $x_k\to x^*$,
the full primal-dual sequence converges to
$(x^*,y^*)\in\cS$.

\medskip
\textup{(iii)} For $k\ge2$, the dual update
\eqref{eq:disc-y} supplies the subgradient
\[
 \eta_k
 :=\frac{\bar y_{k-1}-y_k}{\beta_{k-1}}
 +A\left(x_{k-1}+\frac{x_k-x_{k-1}}{\tau}\right)
 \in\partial G(y_k).
\]
Subtracting $Ax_k$ gives
\[
 \eta_k-Ax_k
 =\frac{\bar y_{k-1}-y_k}{\beta_{k-1}}
 +\frac{1-\tau}{\tau}A(x_k-x_{k-1}).
\]
By \eqref{eq:dybd} and \eqref{eq:bars},
both $\{y_k\}_{k\ge1}$ and
$\{\bar y_k\}_{k\ge1}$ are bounded.
Consequently, for some constant $C_1>0$,
\begin{equation}\label{eq:ddualrate}
 \begin{split}
 \norm{\eta_k-Ax_k}
 &\le
 \frac{C_1}{\beta_{k-1}}
 +\frac{1-\tau}{\tau}\norm{A}
 \norm{x_k-x_{k-1}}\\
 &\le
 \frac{C_1}{(1-\tau)\beta_k}
 +\frac{1-\tau}{\tau}\norm{A}
 \norm{x_k-x_{k-1}}\\
 &=\mathcal O\left(\frac1{\sqrt{\beta_k}}\right),
 \end{split}
\end{equation}
where we used \eqref{eq:dscale},
\eqref{eq:dxrate}, and $\beta_k\to+\infty$.

For the primal component, part~(i),
\eqref{eq_saddle_dis}, and Lipschitz continuity
of $\nabla F$ give
\begin{equation}\label{eq:dprate}
 \begin{split}
 \norm{\nabla F(x_k)+A^\trans y_k}
 &=
 \norm{\nabla F(x_k)-\nabla F(x^*)
       +A^\trans(y_k-y^*)}\\
 &\le
 L_F\norm{x_k-x^*}
 +\norm{A^\trans(y_k-y^*)}\\
 &=\mathcal O\left(\frac1{\sqrt{\beta_k}}\right).
 \end{split}
\end{equation}
Finally, by the definition of distance and
the membership $\eta_k\in\partial G(y_k)$,
\[
 \operatorname{dist}\left(
 0,\,
 \begin{pmatrix}
 \nabla F(x_k)+A^\trans y_k\\
 \partial G(y_k)-Ax_k
 \end{pmatrix}
 \right)\le
 \left\|
 \begin{pmatrix}
 \nabla F(x_k)+A^\trans y_k\\
 \eta_k-Ax_k
 \end{pmatrix}
 \right\|=
 \mathcal O\left(\frac1{\sqrt{\beta_k}}\right),
\]
where the last estimate follows from
\eqref{eq:ddualrate} and \eqref{eq:dprate}.
This proves (iii).
\end{proof}

Condition \eqref{eq:dscale} gives $\beta_k\le\beta_0/(1-\tau)^k$. Maximal geometric growth yields a linear gap bound, while polynomial scales give the following rates.

\begin{corollary}\label{cor:dscale}
Let Assumption~\ref{ass:disc} hold, and let
$\{(x_k,y_k)\}_{k\ge0}$ be generated by
Algorithm~\ref{alg:pd}. For any $\beta_0>0$,
the following conclusions hold.
\begin{enumerate}
\item[(i)]
If $\beta_k=\beta_0/(1-\tau)^k$, then
\[
 \Gap(x_k,y_k)
 =\mathcal O\big((1-\tau)^k\big).
\]

\item[(ii)]
If
$\beta_k=\beta_0(1+k/c)^p$, where
$p>0$ and $c\ge p/\tau$, then
\[
 \Gap(x_k,y_k)=\mathcal O(1/k^p).
\]
\end{enumerate}
\end{corollary}

\begin{proof}
For the geometric scale,
$(1-\tau)\beta_k=\beta_{k-1}$.
For the polynomial scale,
\[
 \frac{\beta_k}{\beta_{k-1}}
 =\left(1+\frac{1}{c+k-1}\right)^p
 \le e^{p/(c+k-1)}
 \le e^\tau
 \le\frac{1}{1-\tau},
\]
where the last inequality follows from
$\log(1-\tau)\le-\tau$.
Thus both choices satisfy \eqref{eq:dscale}
and tend to infinity. The conclusions follow
from Theorem~\ref{thm:dconv}(i).
\end{proof}

 \begin{remark}
Corollary~\ref{cor:dscale} gives linear gap
convergence under one-sided strong convexity,
compared with the polynomial estimates in
\cite{HeHFCNSNS,DingFVCOAP}, and without
the dual smoothness and strong convexity
required in \cite{HeAuto27}.
For $\mu<L_F$ and $s=1/L_F$,
$
 \Gap(x_k,y_k)
 =\mathcal O\left(
 \left(1-\sqrt{\tfrac{\mu}{L_F}}\right)^k
 \right),
$
recovering the strongly convex Nesterov rate
\cite{NesterovBook}; when $A=0$, the primal
iteration reduces to that method.
For affine $G$, the problem reduces to
linear equality constrained optimization as studied
in \cite{HeHFTAC,HeHFAUTO26},
whereas our analysis allows general
nonsmooth convex dual potentials.
It also ensures convergence of the full
sequence with possibly nonunique dual
solutions and a saddle residual bound
of $\mathcal O((\sqrt{1-\tau})^k)$,
without asserting this rate for the dual
distance.
 \end{remark}

\section{Numerical Experiments}\label{sec:num}
In this section, we compare the proposed dynamics and algorithm with existing methods using the primal-dual gap, squared distances to a reference saddle point, and saddle residual. All comparisons use identical data and initial conditions.

 \begin{figure}[!htb]
 \centering
 \begin{minipage}[t]{0.44\textwidth}
  \centering
  \includegraphics[width=\linewidth]{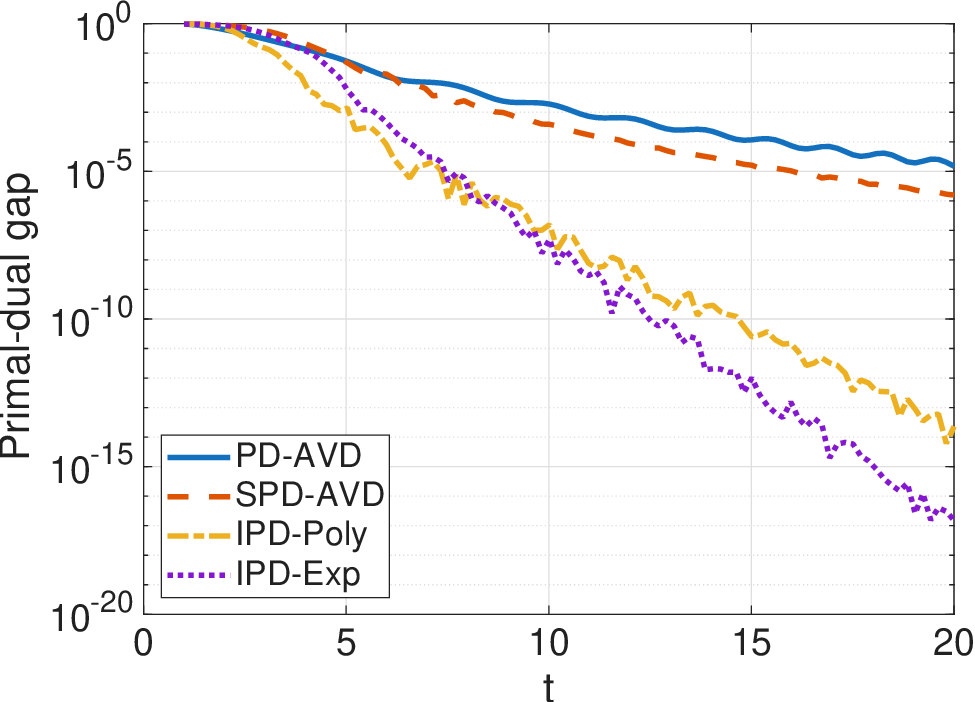}
  \par\vspace{1mm}
  {\small (a) Primal-dual gap}
 \end{minipage}\hspace{0.03\textwidth}%
 \begin{minipage}[t]{0.44\textwidth}
  \centering
  \includegraphics[width=\linewidth]{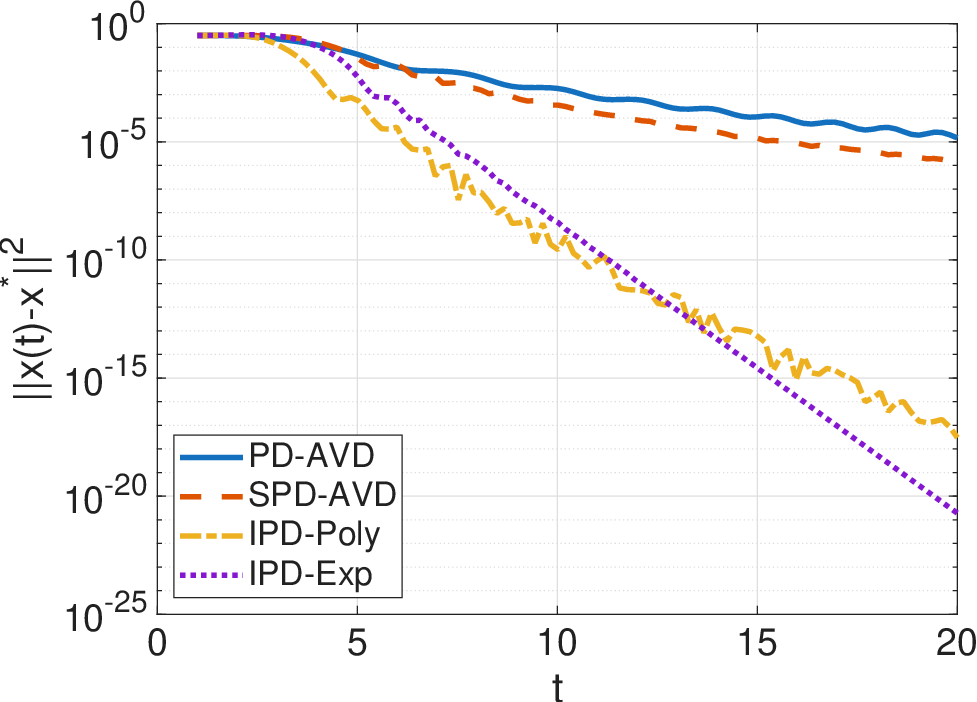}
  \par\vspace{1mm}
  {\small (b) Squared primal distance}
 \end{minipage}

 \par\vspace{2mm}
 \begin{minipage}[t]{0.44\textwidth}
  \centering
  \includegraphics[width=\linewidth]{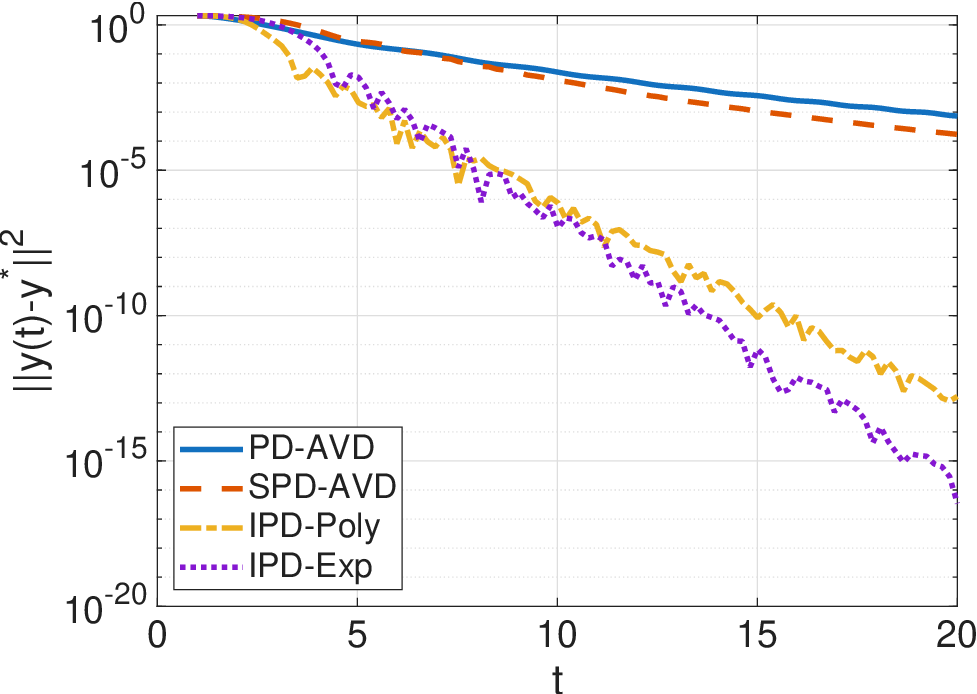}
  \par\vspace{1mm}
  {\small (c) Squared dual distance}
 \end{minipage}\hspace{0.03\textwidth}%
 \begin{minipage}[t]{0.44\textwidth}
  \centering
  \includegraphics[width=\linewidth]{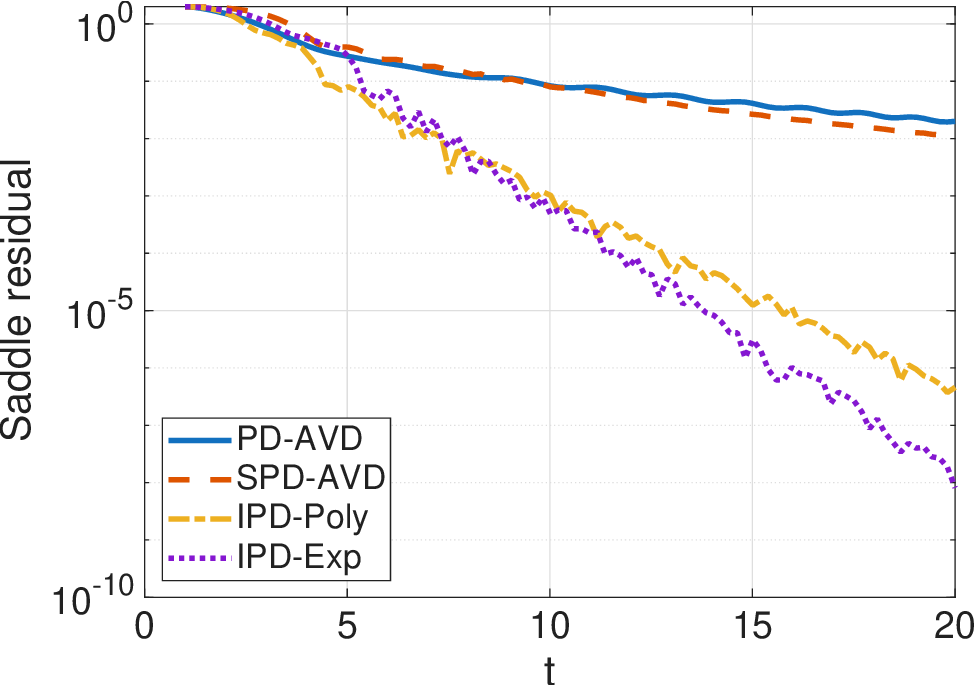}
  \par\vspace{1mm}
  {\small (d) Saddle residual}
 \end{minipage}

 \caption{Comparison of PD-AVD, SPD-AVD, IPD-Poly,
 and IPD-Exp for signal reconstruction.}
 \label{fig:signal-dynamics}
\end{figure}

\subsection{Signal reconstruction}
We consider the reconstruction of a  signal
from noisy measurements of its adjacent differences.
Let $z^\dagger\in\R^n$ denote the underlying signal and
$D\in\R^{(n-1)\times n}$ the first-order difference matrix,
defined by $(Dz)_i=z_{i+1}-z_i$.
The observations are generated as
\[
 b=Dz^\dagger+\sigma_{\mathrm{noise}}\xi,
 \qquad \xi\sim\mathcal N(0,I).
\]
We recover the signal by solving the Tikhonov-regularized
least-squares problem
\[
 \min_{z\in\R^n}
 \frac12\norm{Dz-b}^2
       +\frac{\lambda}{2}\norm{z}^2.
\]
The data-fitting term measures consistency with the
difference observations, while the regularization term
stabilizes reconstruction and selects a zero-mean
representative, since the observations do not determine
the constant component of the signal. Setting $\mu=1/\lambda$, solving the regularized
reconstruction problem is equivalent to solving
the following saddle point problem:
\[
 \min_{x\in\R^n}\max_{y\in\R^n}
 \cL(x,y)
 =\frac{\mu}{2}\norm{x}^2+\ip{x}{y}
  -\frac12\norm{Dy-b}^2.
\]
Indeed, minimizing $\cL(x,y)$ over $x$ gives the
negative of the regularized reconstruction objective.
Thus, $y^*$ is the reconstructed signal and
$x^*=-\lambda y^*$.
Here, $F$ is $\mu$-strongly convex, whereas
$G(y)=\frac12\norm{Dy-b}^2$ is smooth but not strongly
convex because $D\mathbf1=0$.
The unique saddle point is computed from
\[
 (D^\trans D+\lambda I)y^*=D^\trans b,
 \qquad x^*=-\lambda y^*.
\]

We take $n=20$, $\lambda=0.5$, and
$\sigma_{\mathrm{noise}}=0.01$, so that $\mu=2$.
The underlying signal is
$
 z_i^\dagger
 =\sin(2\pi s_i)+\frac12\sin(6\pi s_i),$ $s_i=\frac{i-1}{n-1}.
$
All dynamical systems use identical initial conditions and are
integrated over $[1,20]$. We compare four dynamical systems. The convex dynamical system in \cite{DingFVCOAP} (PD-AVD; also see \cite{ZengLCIFAC}) uses damping $\alpha/t$ and extrapolation $\theta t$, with $\alpha=6$ and $\theta=1/4$. Its scaled counterpart \cite{HeHFAMO} (SPD-AVD) uses the same parameters and $\sigma(t)=0.1t^2$ multiplying both forces. The proposed inertial primal dual dynamic \eqref{eq:ode} use scales
\[
 \beta_{\mathrm{poly}}(t)=0.1t^4\quad\text{(IPD-Poly)},
 \qquad
 \beta_{\mathrm{exp}}(t)=0.1e^{\sqrt2(t-1)}\quad\text{(IPD-Exp)}.
\]
For the polynomial scale, the condition $\dot\beta(t)\le\sqrt\mu\beta(t)$ holds for $t\ge2\sqrt2$; the theoretical estimates therefore apply after this initial interval. The exponential scale satisfies this condition with equality throughout the simulation. The respective gap bounds are $\mathcal O(1/t^2)$, $\mathcal O(1/t^4)$, $\mathcal O(1/t^4)$, and $\mathcal O(e^{-\sqrt2(t-1)})$.  We report the primal-dual gap  $\Gap(x(t),y(t))$, the squared
primal distance $\norm{x(t)-x^*}^2$ and dual distance $\norm{y(t)-y^*}^2$, and the saddle residual $
 \sqrt{
 \norm{\mu x(t)+y(t)}^2+
 \norm{D^\trans(Dy(t)-b)-x(t)}^2}.
$
 
 Figure~\ref{fig:signal-dynamics} shows faster reduction of all four errors for the proposed systems in this experiment, especially IPD-Exp. The observations are consistent with the theoretical bounds; the quadratic model can exhibit faster decay than these upper bounds. The dual-distance curve also illustrates trajectory convergence, for which the general theory does not give a distance rate.

\subsection{Image deblurring}

\begin{figure}[htbp]
 \centering
 \includegraphics[width=0.80\textwidth]{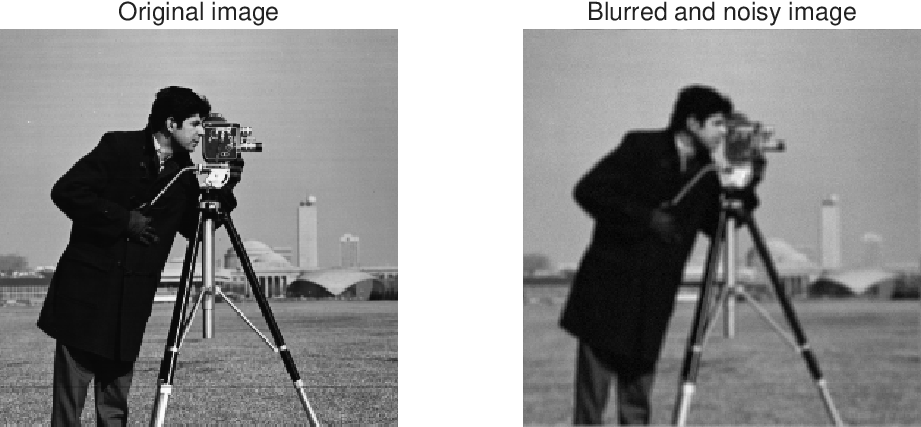}
 \caption{Original image and blurred noisy observation.}
 \label{fig:deblur-data}
\end{figure}

\begin{figure}[!htb]
 \centering
 \begin{minipage}[t]{0.44\textwidth}
  \centering
  \includegraphics[width=\linewidth]{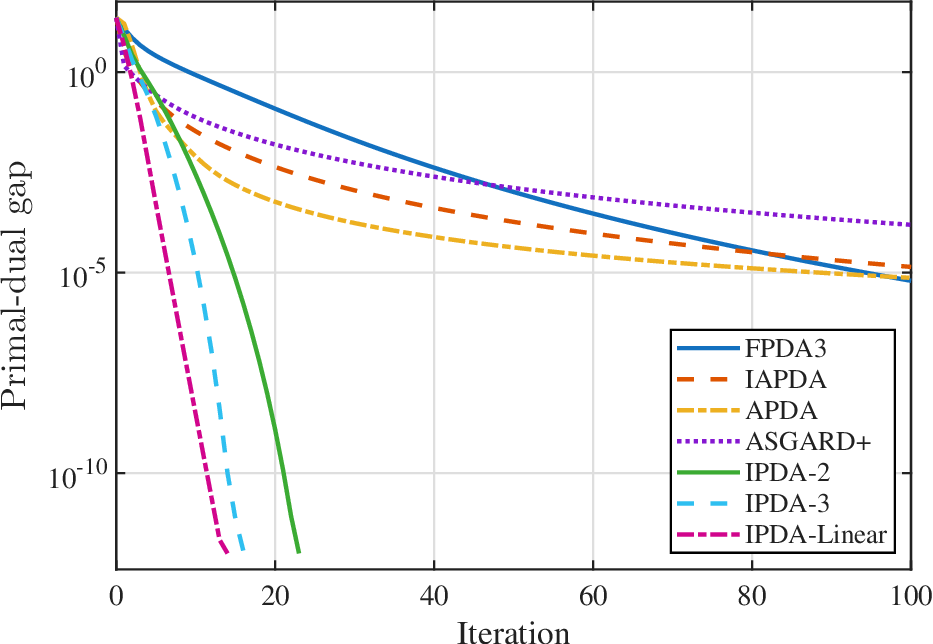}
  \par\vspace{1mm}
  {\small (a) Primal-dual gap}
 \end{minipage}\hspace{0.03\textwidth}%
 \begin{minipage}[t]{0.44\textwidth}
  \centering
  \includegraphics[width=\linewidth]{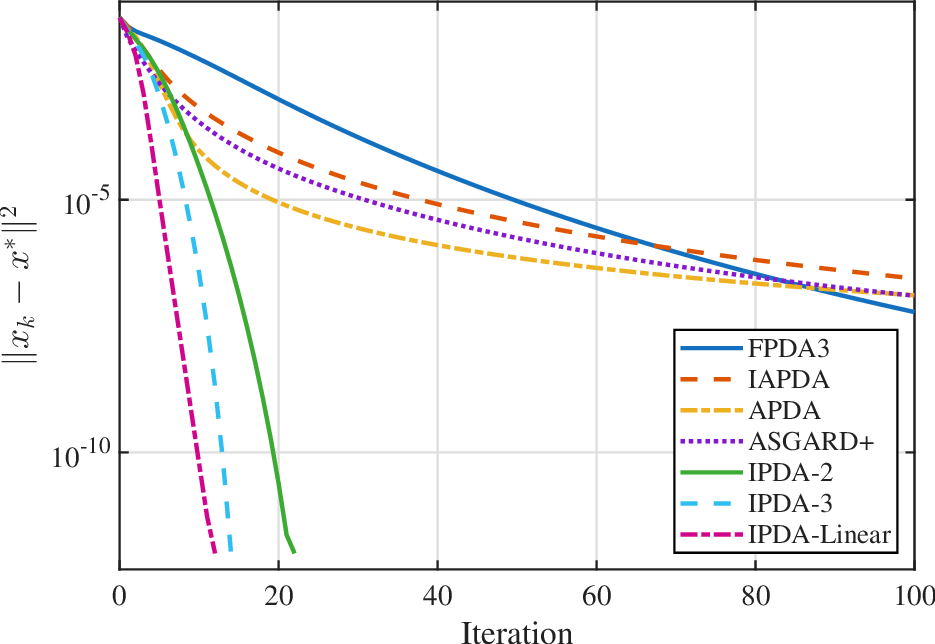}
  \par\vspace{1mm}
  {\small (b) Squared primal distance}
 \end{minipage}

 \par\vspace{2mm}
 \begin{minipage}[t]{0.44\textwidth}
  \centering
  \includegraphics[width=\linewidth]{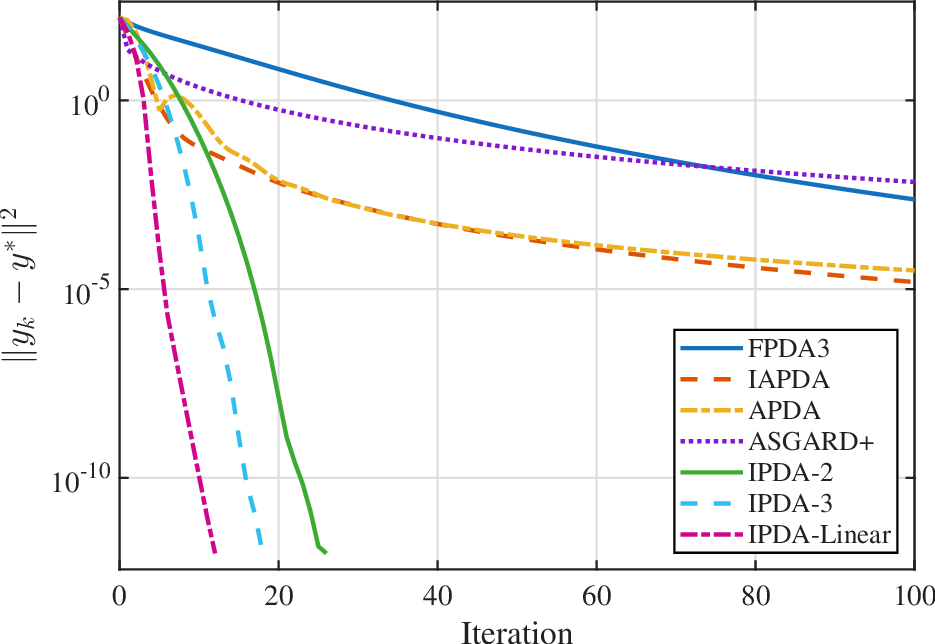}
  \par\vspace{1mm}
  {\small (c) Squared dual distance}
 \end{minipage}\hspace{0.03\textwidth}%
 \begin{minipage}[t]{0.44\textwidth}
  \centering
  \includegraphics[width=\linewidth]{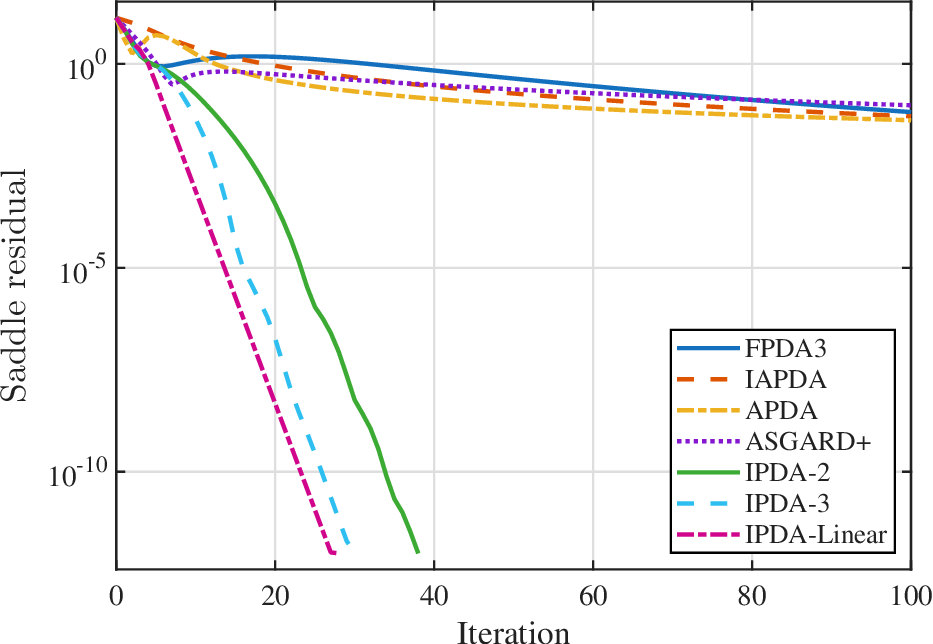}
  \par\vspace{1mm}
  {\small (d) Saddle residual}
 \end{minipage}

 \caption{Convergence comparison for image deblurring.}
 \label{fig:deblur-convergence}
\end{figure}

\begin{figure}[htbp]
 \centering
 \includegraphics[width=0.98\textwidth]{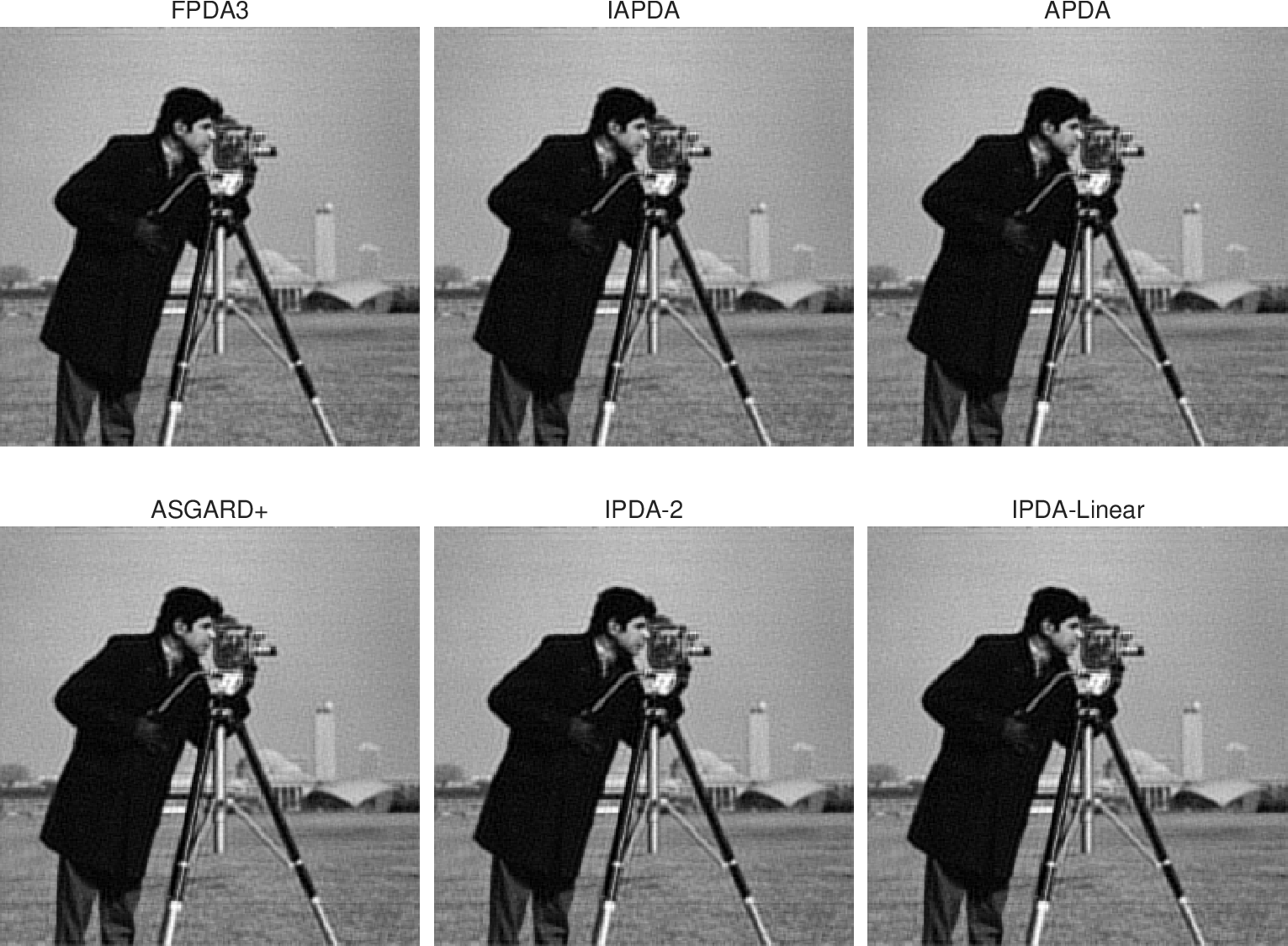}
 \caption{Reconstructed images after 100 iterations.}
 \label{fig:deblur-restoration}
\end{figure}

We consider the quadratically regularized least-squares model for image deblurring:
\[
 \min_{u\in\R^n}
 \frac12\norm{Hu-b}^2
 +\frac1{2\mu}\norm{Du}^2,
 \qquad \mu>0.
\]
Here $u$ is the image, $H$ is a periodic blur matrix, and $D:\R^n\to\R^{2n}$ is the periodic discrete image-gradient operator. The original image is converted to grayscale and normalized to $[0,1]$. We use $\mu=100$, a $4\times4$ averaging blur, and the observation $b=Hu_{\mathrm{true}}+\varepsilon$, with independent Gaussian noise of standard deviation $0.01$.

Exchanging the primal and dual roles in the standard saddle formulation and changing the sign yields the equivalent problem 
\[ 
     \min_{x\in\R^{2n}}\max_{y\in\R^n} \cL(x,y) =\frac{\mu}{2}\norm{x}^2 -\ip{D^\trans x}{y} -\frac12\norm{Hy-b}^2. 
\] 
Thus, $F(x)=\mu\norm{x}^2/2$, $G(y)=\norm{Hy-b}^2/2$, and $A=-D^\trans$. The variable $y$ represents the reconstructed image, while $x$ is an auxiliary variable. We choose a periodic averaging blur with a nontrivial kernel, so that $G$ is smooth and convex but not strongly convex.

We compare Algorithm~\ref{alg:pd} (IPDA) with the fast primal-dual algorithm \cite[Algorithm~3]{DingFVCOAP} (FPDA3), the inertial accelerated primal-dual algorithm \cite[Algorithm~1, Option~1]{HeHFCNSNS} (IAPDA), the accelerated primal-dual algorithm of Chambolle and Pock \cite{ChambollePJMIV} (APDA), and the accelerated smoothed gap reduction algorithm \cite{TranDinhOL} (ASGARD+). Let $a=\norm{A}$ and $L_G=\norm{H}^2$. For FPDA3, we use $\gamma=3/4$, $t_k=(k+5)/6$, and stepsizes $\sigma=1/(2\mu)$ and $\rho=1/(2L_G)$. For IAPDA, we apply the method to $\min_y\max_x[-\cL(x,y)]$, with proximal treatment of both potentials, and choose $\alpha_{\mathrm H}=0.45\mu/a^2$, $\beta_{\mathrm H}=2/\mu$, and $t_1=2$. APDA uses initial stepsizes $\tau_0^{\mathrm C}=\sigma_0^{\mathrm C}=0.9/a$. ASGARD+ uses initial acceleration parameter $1$ and smoothing parameter $a^2/\mu$. The remaining parameters follow the corresponding accelerated update rules. For the proposed inertial accelerated primal-dual algorithm (Algorithm \ref{alg:pd},  IPDA), we set $s=1/(2\mu)$ and $\tau=\sqrt{\mu s}=1/\sqrt2$, and consider
\[
 \beta_k^{(2)}
 =\beta_0\left(1+\frac{k}{c}\right)^2,
 \qquad
 \beta_k^{(3)}
 =\beta_0\left(1+\frac{k}{c}\right)^3,
 \qquad
 \beta_k^{(\mathrm{Linear})}
 =\frac{\beta_0}{(1-\tau)^k},
\]
with $\beta_0=1$ and $c=3/\tau$.
The three variants are denoted by
IPDA-2, IPDA-3, and IPDA-Linear, respectively,
and yield gap estimates of
$\mathcal O(1/k^2)$, $\mathcal O(1/k^3)$,
and $\mathcal O((1-\tau)^k)$.
 All methods use the same observation
and initial pair.
 
The reference saddle point is computed from
\[
 \left(H^\trans H+\frac1\mu D^\trans D\right)y^*
 =H^\trans b,
 \qquad
 x^*=\frac1\mu Dy^*.
\]
Although $H$ is singular, it preserves constant
images, while the kernel of $D$ consists of
constant images. Consequently,
$\ker H\cap\ker D=\{0\}$, and the reference
saddle point is unique. We evaluate the primal-dual gap $\Gap(x_k,y_k)$,
the squared distances to the reference
saddle point, and the saddle residual:
$\sqrt{
 \norm{\mu x_k-Dy_k}^2+
 \norm{H^\trans(Hy_k-b)+D^\trans x_k}^2}.
$ Figure~\ref{fig:deblur-data} shows the original image and
the blurred noisy observation. Figure~\ref{fig:deblur-convergence} compares the four errors against completed iterations and  Figure~\ref{fig:deblur-restoration} compares reconstructions after 100 iterations. The results show that the proposed algorithms converge faster in this experiment, particularly IPDA-Linear, consistent with Theorem~\ref{thm:dconv} and Corollary~\ref{cor:dscale}.


\subsection{Elastic-net regularization}

We consider the elastic-net problem
\begin{equation}
 \min_{x\in\mathbb R^n} \frac12\|Bx-b\|^2+\frac\mu2\|x\|^2+\lambda\|x\|_1,
 \label{eq:enet-problem}
\end{equation}
where $B\in\mathbb R^{m\times n}$ and $\mu,\lambda>0$.
Its saddle-point formulation is $\min_x\max_y\mathcal L(x,y)$, with
\[
 \mathcal L(x,y)=F(x)+\langle x,y\rangle-\delta_C(y),\qquad
 F(x)=\frac12\|Bx-b\|^2+\frac\mu2\|x\|^2,\qquad
 C=[-\lambda,\lambda]^n,
\]
where $\delta_C$ is the indicator of $C$. Thus the coupling operator is
$A=I$, and the dual potential is convex but not strongly convex.
We set $m=2000$, $n=5000$, and $\mu=1$. A matrix with independent
standard Gaussian entries is rescaled so that $\|B\|_2^2=9$.
Since $m<n$, the exact strong convexity and smoothness constants of $F$
are $\mu=1$ and $L_F=10$, respectively. The vector $x_{\mathrm{true}}$ has $500$
uniformly selected nonzero entries with independent standard Gaussian
values, and
\[
 b=Bx_{\mathrm{true}}+\sigma\xi,\qquad
 \sigma=0.01\|Bx_{\mathrm{true}}\|/\sqrt m,\qquad
 \xi\sim\mathcal N(0,I_m),\qquad
 \lambda=0.05\|B^\top b\|_\infty.
\]

We compare our IPDA-Linear (Algorithm \ref{alg:pd}) with IAPDA~\cite{HeHFCNSNS},
APAPC~\cite{CondatSR26}, and ASGARD+~\cite{TranDinhOL}.
All primal, dual, and auxiliary vector states are initialized at zero. IPDA-Linear uses the full step and geometric scaling
\[
 s=L_F^{-1},\qquad \tau=\sqrt{\mu/L_F}=\sqrt{0.1},\qquad
 \beta_k=(1-\tau)^{-k}\quad (\beta_0=1).
\]
IAPDA uses Option~1 on the equivalent dual saddle formulation, with
interchanged variables, $K=-I$, the full proximal mapping of $F$, and
$(\alpha_{\mathrm H},\beta_{\mathrm H},t_1)=(0.45,2,2)$.
APAPC splits the least-squares term from the quadratic regularizer;
its smoothness constant is $9$.
ASGARD+ uses its strongly convex parameter rule with initial acceleration
and smoothing parameters equal to one and smoothing center zero.
The remaining updates follow the cited rules.

Let $(x^*,y^*)$ be a reference saddle point. We measure the relative
primal-dual gap
\[ g_k=\frac{\Delta_k}{\Delta_0},\qquad
 \Delta_k=\mathcal L(x_k,y^*)-\mathcal L(x^*,y_k).
\]
For $e_k=x_k-x^*$ and $y_k\in C$, a stable evaluation is
\[
 \Delta_k=\frac12\|Be_k\|^2+\frac\mu2\|e_k\|^2
 +\sum_{i=1}^n |x_i^*|
   \bigl[\lambda-\operatorname{sign}(x_i^*)(y_k)_i\bigr].
\]
To assess stationarity and dual consistency in addition to this gap,
we also record the relative projected KKT residual
\begin{equation}
 r_k=\frac{\left(\|\nabla F(x_k)+y_k\|^2+
          \|y_k-\Pi_C(y_k+x_k)\|^2\right)^{1/2}}
          {\|\nabla F(0)\|},
 \label{eq:enet-residual}
\end{equation}
where $\Pi_C$ denotes Euclidean projection and the projection step is one.
The reference solution is computed independently by accelerated proximal
gradient iterations and active-set polishing, with a relative coordinate
KKT residual below $10^{-10}$.

Table~\ref{tab:enet} reports the iteration counts and corresponding
running times required to attain $g_k\le10^{-7}$ and $r_k\le10^{-5}$.
The results are presented separately for the two criteria, allowing
the methods to be compared in terms of both primal-dual gap reduction
and satisfaction of the KKT conditions. Figure~\ref{fig:enet} shows the evolution of these two errors
during the first $500$ iterations. IPDA-Linear produces the most rapid
overall decrease in both panels, while APAPC also exhibits fast error
reduction. In comparison, IAPDA and ASGARD+ show a more gradual decrease
over the displayed interval. The convergence curves are consistent
with the iteration counts reported in Table~\ref{tab:enet}.
On this instance, our method, IPDA-Linear, reaches both targets with
the fewest iterations and the shortest running time.
The gains are particularly pronounced over IAPDA and ASGARD+,
while improvements over APAPC are also observed under both criteria.

\begin{table}[htbp]
 \centering
\caption{Computational performance of the compared algorithms
on the elastic-net problem.} \label{tab:enet}
 \begingroup
 \renewcommand{\arraystretch}{1.3}
 \begin{tabular*}{0.9\linewidth}{@{\extracolsep{\fill}}lrrrr@{}}
  \toprule
  &\multicolumn{2}{c}{$g_k\le10^{-7}$}
  &\multicolumn{2}{c}{$r_k\le10^{-5}$}\\
  \cmidrule(lr){2-3}\cmidrule(lr){4-5}
  Method & Iterations & Time (s) & Iterations & Time (s)\\
  \midrule
  IPDA-Linear & 18   & 0.065  & 27   & 0.075\\
  IAPDA       & 251  & 6.722  & 1352 & 10.295\\
  APAPC       & 41   & 0.148  & 53   & 0.162\\
  ASGARD+     & 3754 & 18.423 & 2174 & 13.090\\
  \bottomrule
 \end{tabular*}
 \endgroup
\end{table}

\begin{figure}[htbp]
 \centering
 \includegraphics[width=0.48\linewidth]{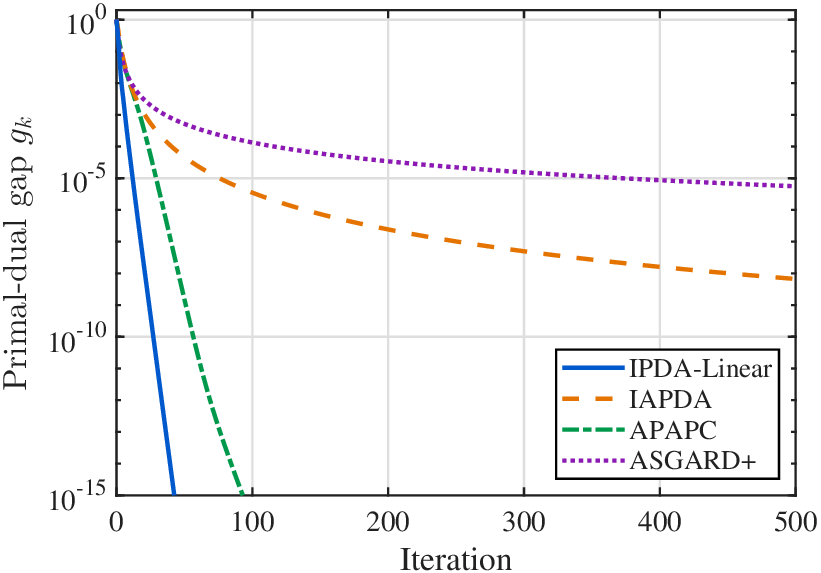}
 \hfill
 \includegraphics[width=0.48\linewidth]{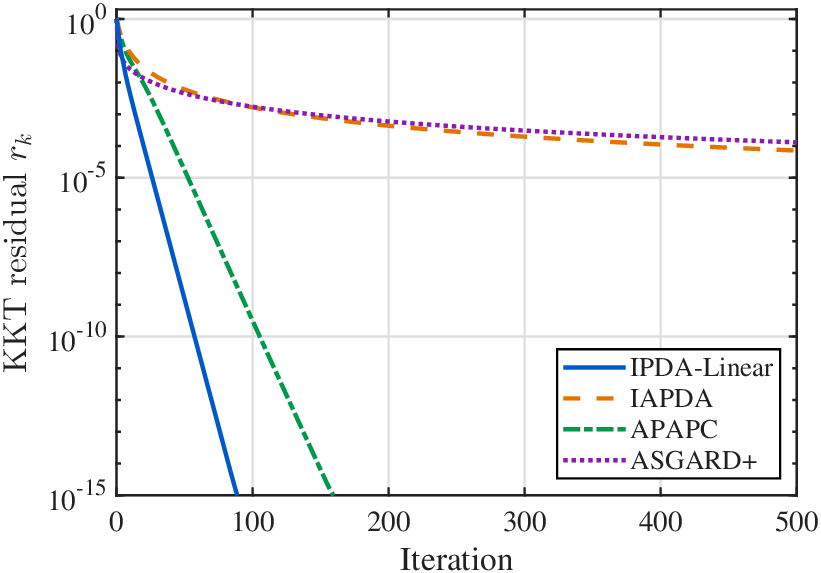}
 \caption{Convergence comparison for primal-dual gap $g_k$ and projected KKT residual $r_k$.}
 \label{fig:enet}
\end{figure}

\section{Conclusion}\label{sec:conclusion}
In this paper, we have developed accelerated primal-dual dynamics and an associated algorithm for bilinear saddle point problems with a strongly convex primal potential and a merely convex dual potential. With admissible exponential and geometric scales, combining primal strong convexity with dual scaling yields exponential gap decay in continuous time and linear gap decay in discrete time. Beyond gap estimates, we establish saddle residual bounds and convergence of the primal-dual trajectory and sequence, even when the dual solution is nonunique. These results complement the polynomial convergence estimates in \cite{DingFVCOAP,HeHFCNSNS,HeFarXiv26} and extend exponential and linear gap guarantees to the one-sided strongly convex setting, without the dual strong convexity required in \cite{HeAuto27}. Numerical experiments illustrate the theoretical results and demonstrate the improved convergence performance of the proposed methods.


\begin{thebibliography}{99}
\setlength{\itemsep}{0pt}   
\bibitem{AttouchMOR}
Attouch, H., Bo\c{t}, R.I., Nguyen, D.-K.: Fast convex optimization via time scale and averaging of the steepest descent. Math. Oper. Res. 50(4), 2633--2665 (2025)

\bibitem{AttouchCRSJO}
Attouch, H., Chbani, Z., Riahi, H.: Fast proximal methods via time scaling of damped inertial dynamics. SIAM J. Optim. 29(3), 2227--2256 (2019)

\bibitem{BaiJSC25}
Bai, J., Chen, Y., Yu, X., Zhang, H.: Generalized asymmetric forward--backward--adjoint algorithms for convex--concave saddle-point problem. J. Sci. Comput. 102(3), 80 (2025)

\bibitem{BaiJNM}
Bai, J., Li, J., Wu, Z.: Several variants of the primal-dual hybrid gradient algorithm with applications. Numer. Math. Theory Methods Appl. 13(1), 176--199 (2020)

\bibitem{BaiORL}
Bai, J., Rao, S.: A new insight on the prediction-correction framework with applications to several first-order methods. Oper. Res. Lett. 68, 107487 (2026)

\bibitem{Balhag}
Balhag, A., Chbani, Z., Riahi, H.: Linear convergence of inertial gradient dynamics with constant viscous damping coefficient and time-dependent rescaling parameter. HAL preprint \href{https://hal.science/hal-02610699}{hal-02610699} (2020)

\bibitem{BotCNMP}
Bo\c{t}, R.I., Csetnek, E.R., Nguyen, D.-K.: Fast augmented Lagrangian method in the convex regime with convergence guarantees for the iterates. Math. Program. 200, 147--197 (2023)

\bibitem{BotNJDE}
Bo\c{t}, R.I., Nguyen, D.-K.: Improved convergence rates and trajectory convergence for primal-dual dynamical systems with vanishing damping. J. Differ. Equ. 303, 369--406 (2021)

\bibitem{BoydFTML}
Boyd, S., Parikh, N., Chu, E., Peleato, B., Eckstein, J.: Distributed optimization and statistical learning via the alternating direction method of multipliers. Found. Trends Mach. Learn. 3(1), 1--122 (2011)


\bibitem{ChambollePJMIV}
Chambolle, A., Pock, T.: A first-order primal-dual algorithm for convex problems with applications to imaging. J. Math. Imaging Vis. 40, 120--145 (2011)

\bibitem{ChambollePAN}
Chambolle, A., Pock, T.: An introduction to continuous optimization for imaging. Acta Numer. 25, 161--319 (2016)

\bibitem{ChangJSC}
Chang, X., Yang, J.: A golden ratio primal-dual algorithm for structured convex optimization. J. Sci. Comput. 87(2), 47 (2021)

\bibitem{ChenLOSJO}
Chen, Y., Lan, G., Ouyang, Y.: Optimal primal-dual methods for a class of saddle point problems. SIAM J. Optim. 24(4), 1779--1814 (2014)

\bibitem{CondatSR26}
Condat, L., Sadiev, A., Richt\'arik, P.: A Nesterov-accelerated primal-dual splitting algorithm for convex nonsmooth optimization. arXiv:2604.09245 (2026)


\bibitem{DengSiam}
Deng, K., Wang, R., Zhu, Z., Zhang, J., Wen, Z.: The augmented Lagrangian methods: overview and recent advances. SIAM Rev. 68(3), 471--576 (2026)

\bibitem{DingFVCOAP}
Ding, K.-W., Fliege, J., Vuong, P.T.: Fast convergence of the primal-dual dynamical system and corresponding algorithms for a nonsmooth bilinearly coupled saddle point problem. Comput. Optim. Appl. 90, 151--192 (2025)

\bibitem{GhadimiFJECC}
Ghadimi, E., Feyzmahdavian, H.R., Johansson, M.: Global convergence of the heavy-ball method for convex optimization. In: 2015 European Control Conference (ECC), pp. 310--315. IEEE (2015)

\bibitem{HeFarXiv26}
He, X., Fang, Y.P.: Fast primal-dual methods for convex-concave bilinear saddle point problems: continuous-time dynamics and discrete algorithms. arXiv preprint \href{https://doi.org/10.48550/arXiv.2606.18724}{arXiv:2606.18724} (2026). 

\bibitem{HeAuto27}
He, X., Fang, Y.P.: Nesterov acceleration for strongly convex-strongly concave bilinear saddle point problems: discrete and continuous-time approaches. Automatica 195, 113322 (2027)

\bibitem{HeHFAUTO26}
He, X., He, D., Fang, Y.P.: Accelerated primal-dual methods for strongly convex objective functions in continuous and discrete time. Automatica 183, 112579 (2026)

\bibitem{HeHFSJCO}
He, X., Hu, R., Fang, Y.P.: Convergence rates of inertial primal-dual dynamical methods for separable convex optimization problems. SIAM J. Control Optim. 59(5), 3278--3301 (2021)

\bibitem{HeHFTAC}
He, X., Hu, R., Fang, Y.P.: ``Second-order primal'' + ``first-order dual'' dynamical systems with time scaling for linear equality constrained convex optimization problems. IEEE Trans. Autom. Control 67(8), 4377--4383 (2022)

\bibitem{HeHFAMO}
He, X., Hu, R., Fang, Y.P.: A second order primal-dual dynamical system for a convex-concave bilinear saddle point problem. Appl. Math. Optim. 89, 30 (2024)

\bibitem{HeHFJCAM}
He, X., Hu, R., Fang, Y.P.: Inertial accelerated augmented Lagrangian algorithms with scaling coefficients to solve exactly and inexactly linearly constrained convex optimization problems. J. Comput. Appl. Math. 460, 116425 (2025)

\bibitem{HeHFCNSNS}
He, X., Huang, N.J., Fang, Y.P.: Non-ergodic convergence rate of an inertial accelerated primal-dual algorithm for saddle point problems. Commun. Nonlinear Sci. Numer. Simul. 140, 108289 (2025)

\bibitem{KhalafiBPMLR}
Khalafi, M., Boob, D.: Accelerated primal-dual methods for convex-strongly-concave saddle point problems. In: Proceedings of the 40th International Conference on Machine Learning. Proc. Mach. Learn. Res., vol. 202, pp. 16250--16270. PMLR (2023)

\bibitem{KovalevGR}
Kovalev, D., Gasnikov, A., Richt\'{a}rik, P.: Accelerated primal-dual gradient method for smooth and convex-concave saddle-point problems with bilinear coupling. In: Advances in Neural Information Processing Systems, vol. 35, pp. 21725--21737 (2022)

\bibitem{LiTac}
Li, W., Zeng, X., Pavel, L.: Primal-dual accelerated mirror-descent method for constrained bilinear saddle-point problems. IEEE Trans. Autom. Control 71(2), 1373--1380 (2026)

\bibitem{LinAcc}
Lin, Z., Li, H., Fang, C.: Accelerated Optimization for Machine Learning: First-Order Algorithms. Springer, Singapore (2020)

\bibitem{LongCnsns}
Long, X.J., Nie, J.L., Gou, Z., Sun, X.K., Li, G.X.: An accelerated preconditioned primal-dual gradient algorithm for structured nonconvex optimization problems. Commun. Nonlinear Sci. Numer. Simul. 153, 109480 (2026)

\bibitem{LuoArx24}
Luo, H.: Accelerated primal-dual proximal gradient splitting methods for convex-concave saddle-point problems. arXiv preprint \href{https://doi.org/10.48550/arXiv.2407.20195}{arXiv:2407.20195} (2024)

\bibitem{LuoOpt}
Luo, H.: A continuous perspective on the inertial corrected primal-dual proximal splitting. Optimization 75(12), 3203--3232 (2026)

\bibitem{LuoJOGO}
Luo, H.: Accelerated primal-dual methods for linearly constrained convex optimization problems. J. Glob. Optim. 95(2), 405--437 (2026)

\bibitem{LuoCMP}
Luo, H., Chen, L.: From differential equation solvers to accelerated first-order methods for convex optimization. Math. Program. 195, 735--781 (2022)

\bibitem{LuoMC}
Luo, H., Zhang, Z.: A unified differential equation solver approach for separable convex optimization: splitting, acceleration and nonergodic rate. Math. Comput. 94(356), 3009--3041 (2025)

\bibitem{NesterovBook}
Nesterov, Y.: Introductory Lectures on Convex Optimization: A Basic Course. Applied Optimization, vol. 87. Kluwer Academic Publishers, Boston (2004)

\bibitem{PolyakCMMP}
Polyak, B.T.: Some methods of speeding up the convergence of iteration methods. USSR Comput. Math. Math. Phys. 4(5), 1--17 (1964)

\bibitem{Siegel}
Siegel, J.W.: Accelerated first-order methods: differential equations and Lyapunov functions. arXiv preprint \href{https://doi.org/10.48550/arXiv.1903.05671}
{arXiv:1903.05671} (2019)

\bibitem{SuBCJMLR}
Su, W., Boyd, S., Cand\`es, E.J.: A differential equation for modeling Nesterov's accelerated gradient method: theory and insights. J. Mach. Learn. Res. 17(153), 1--43 (2016)

\bibitem{SunAAAI}
Sun, T., Yin, P., Li, D., Huang, C., Guan, L., Jiang, H.: Non-ergodic convergence analysis of heavy-ball algorithms. Proc. AAAI Conf. Artif. Intell. 33(1), 5033--5040 (2019)

\bibitem{SunJOTA}
Sun, X.K., He, L., Long, X.J.: Inertial primal-dual dynamics with Hessian-driven damping and Tikhonov regularization for convex-concave bilinear saddle point problems. Optimization \href{https://doi.org/10.1080/02331934.2025.2578403}
{Doi:10.1080/02331934.2025.2578403} (2025)

\bibitem{TeschlBook}
Teschl, G.: Ordinary Differential Equations and Dynamical Systems. Graduate Studies in Mathematics, vol. 140. American Mathematical Society, Providence, RI (2012)

\bibitem{ThekPMLR}
Thekumparampil, K.K., He, N., Oh, S.: Lifted primal-dual method for bilinearly coupled smooth minimax optimization. In: Proceedings of the 25th International Conference on Artificial Intelligence and Statistics. Proc. Mach. Learn. Res., vol. 151, pp. 4281--4308. PMLR (2022)

\bibitem{TranDinhOL}
Tran-Dinh, Q.: A unified convergence rate analysis of the accelerated smoothed gap reduction algorithm. Optim. Lett. 16(4), 1235--1257 (2022)

\bibitem{TranMPC}
Tran-Dinh, Q., Alacaoglu, A., Fercoq, O., Cevher, V.: An adaptive primal-dual framework for nonsmooth convex minimization. Math. Program. Comput. 12(3), 451--491 (2020)

\bibitem{WangArx}
Wang, Z., Peypouquet, J.: Inertial primal dual dynamics with Hessian-driven damping for saddle point problems. arXiv preprint \href{https://doi.org/10.48550/arXiv.2607.26235}{arXiv:2607.26235} (2026).

\bibitem{WibisonoWJPNAS}
Wibisono, A., Wilson, A.C., Jordan, M.I.: A variational perspective on accelerated methods in optimization. Proc. Natl. Acad. Sci. USA 113(47), E7351--E7358 (2016)

\bibitem{WilsonRJJMLR}
Wilson, A.C., Recht, B., Jordan, M.I.: A Lyapunov analysis of accelerated methods in optimization. J. Mach. Learn. Res. 22(113), 1--34 (2021)

\bibitem{ZengLCIFAC}
Zeng, X., Dou, L., Chen, J.: Accelerated first-order continuous-time algorithm for solving convex-concave bilinear saddle point problem. IFAC-PapersOnLine 53(2), 7362--7367 (2020)

\bibitem{ZhangAMO}
Zhang, H., Sun, X., Li, S., Teo, K.L.: Inertial accelerated primal-dual algorithms for non-smooth convex optimization problems with linear equality constraints. Appl. Math. Optim. 93(3), 81 (2026)


\end{thebibliography}
\end{document}